\documentclass[11pt,reqno]{amsart}

\usepackage{amssymb}
\usepackage{xcolor}
\usepackage{dsfont}
\textheight=8.5in \numberwithin{equation}{section}

\newtheorem{theorem}{Theorem}[section]

\newtheorem{proposition}[theorem]{Proposition}
\newtheorem{lemma}[theorem]{Lemma}
\newtheorem{corollary}[theorem]{Corollary}
\theoremstyle{definition}
\newtheorem{definition}[theorem]{Definition}
\newtheorem{remark}[theorem]{Remark}

\def\N {\mathbb N}
\def\R {\mathbb R}

\def\S {\mathcal S}
\def\Rd {{\mathbb R}^N}
\def\Sd {{\mathcal S'}({\mathbb R}^N)}
\def\Rz{\Rd \backslash \{0\}}
\def\Dd {{\mathcal D'}(\Omega_m)}
\def\D {{\mathcal D}(\Omega_m)}
\def\X {{\mathcal X}_{m,\gamma}}

\def\B {{\mathcal B}_{m,\gamma,M}}
\def\Bs {{\mathcal B}_{m,\gamma,M}^\star}

\def\Bg {{\mathcal B}_{M}^{\gamma+m}}
\def\Bt {\widetilde{\B}}
\def\Bts {\widetilde{\Bs}}
\def\Wd {{\mathcal W}^{\gamma+m}}

\makeatletter
\def\@setcopyright{}
\def\serieslogo@{}
\makeatother

\begin{document}

\author[H. Mouajria]{Hattab MOUAJRIA}
\address{Universit\'e de Carthage, Institut Pr\'eparatoire aux \'etudes d'ing\'enieurs de Nabeul,
Campus Universitaire, Merazka, 8000 Nabeul, Tunisie}
\email{mouajria.hattab@gmail.com}
\author[S. Tayachi]{Slim Tayachi}
\address{Universit\'e de Tunis El Manar, Facult\'e des Sciences de
Tunis, D\'epartement de Math\'ematiques, Laboratoire \'Equations aux
d\'eriv\'ees partielles LR03ES04, 2092 Tunis, Tunisie}
\email{slim.tayachi@fst.rnu.tn}
\author[F. B. Weissler]{Fred B. Weissler}
\address{Universit\'e Paris 13, CNRS UMR 7539 LAGA, 99,
Avenue Jean-Baptiste Cl\'ement, 93430 Villetaneuse,
France}\email{weissler@math.univ-paris13.fr}

\title[The nonlinear heat equation with absorption]{Large time behavior of  solutions to the nonlinear heat equation with absorption with highly singular antisymmetric initial values}
\date{\today}

\begin{abstract}
 In this paper we study global well-posedness and long time asymptotic behavior of solutions to the nonlinear heat equation with absorption,
$ u_t - \Delta u +  |u|^\alpha u =0$,
where $u=u(t,x)\in \R, $ $(t,x)\in (0,\infty)\times\Rd$ and  $\alpha>0$.
We focus particularly on highly singular initial values which are antisymmetric with respect to the variables
$x_1,\; x_2,\; \cdots,\; x_m$ for some $m\in \{1,2, \cdots, N\}$, such as
$u_0 = (-1)^m\partial_1\partial_2 \cdots \partial_m|\cdot|^{-\gamma} \in \Sd$, $0 < \gamma < N$.
In fact, we show global well-posedness  for initial data bounded in an appropriate sense by  $u_0$, for any $\alpha>0$.

Our approach is to study  well-posedness and  large time behavior on sectorial domains
of the form $\Omega_m = \{x \in \R^N : x_1, \cdots, x_m > 0\}$, and then to extend the results by reflection to
solutions on $\R^N$ which are antisymmetric.
We show that the large time behavior depends on the relationship between  $\alpha$ and $2/(\gamma+m)$, and we consider
all three cases, $\alpha$ equal to, greater than, and less than $2/(\gamma+m)$.  Our results include, among others,
new examples of self-similar and asymptotically self-similar solutions.

\end{abstract}

\subjclass[2000]{Primary 35K05, 35B40, 35B06, 35B30; Secondary 35B60, 47A20}
\keywords{Heat equation, absorption, complexity, sectorial domains, asymptotic behavior,
                     antisymmetric solutions, scaling}
\maketitle

\section{Introduction}
\label{intro}

 In this paper we study the long time behavior of solutions to the nonlinear heat equation with absorption,
\begin{equation}
\label{nheq}%
u_t - \Delta u +  |u|^\alpha u =0 ,
\end{equation}
where $u=u(t,x)\in \R, $ $(t,x)\in (0,\infty)\times\Rd$ and  $\alpha>0$,
which are antisymmetric with respect to the variables
$x_1,\; x_2,\; \cdots,\; x_m$ for some $m\in \{1,2, \cdots, N\}$.  Our goal is to see how some
well-known results \cite{CDEW,CDW,CDWMult,GVL} for the long time behavior of solutions to \eqref{nheq}
carry over with the additional hypothesis of antisymmetry.  For example, some of the results in
the cited works concern positive solutions.  We will see that these results have analogues for
antisymmetric solutions which are positive on an appropriate sector in $\Rd$. In particular, these solutions
are not positive on $\Rd$. Moreover, in many cases
the range of allowable powers $\alpha>0$ will be larger with the additional hypothesis of antisymmetry
than without.  Also, the condition of antisymmetry allows consideration of a class of highly singular initial values.

Our previous paper \cite{MTW} considered the {\it linear} heat equation on $\Rd$ with antisymmetric solutions.
The results and the theoretical framework from \cite{MTW} were applied to the
 nonlinear heat equation with source term
\begin{equation}
\label{nheqs}%
u_t - \Delta u -  |u|^\alpha u =0 ,
\end{equation}
in \cite{TW2}.  In the current paper, these ideas are applied to \eqref{nheq}.
We mention that this approach was earlier developed in \cite{TW}
where solutions to \eqref{nheqs}
with antisymmetric initial values of the form
$u_0 = (-1)^m\partial_1\partial_2 \cdots \partial_m\delta$
were studied.  In the current paper, as in \cite{MTW,TW2}, initial values of the form
$u_0 = (-1)^m\partial_1\partial_2 \cdots \partial_m|\cdot|^{-\gamma}$, for some $0 < \gamma < N$,
are considered.

In order to state our results precisely, we begin by recalling the definition of an antisymmetric function.
\begin{definition}
Let $m\in \{1,2, \cdots, N\}$.
A function $f:\Rd\to\R$ is antisymmetric with respect to $x_1, \cdots, x_m$ if it satisfies
\begin{equation}
\label{antisym}
        T_1 f=T_2 f=\cdots=T_m f= - f,
\end{equation}
where $T_i$ , $i\in \{1, 2, \cdots, N\}$, denote the operator
\[
      [T_i f](x_1,\cdots, x_{i-1}, x_i, x_{i+1}, \cdots, x_N)=      f(x_1,\cdots, x_{i-1}, -x_i, x_{i+1}, \cdots, x_N).
\]
We denote the set of functions antisymmetric with respect to $x_1, \cdots, x_m$  by
\begin{equation}
\label{antisymspc}
        \mathcal{A}=\mathcal{A}_m = \{f:\Rd\to\R; \; f\ {\rm satisfies}\ \eqref{antisym}\}.
\end{equation}
\end{definition}
A function on $\Rd$ which is antisymmetric with respect to
$x_1,\; x_2,\; \cdots,\; x_m$, for some $m\in \{1,2, \cdots, N\}$, is determined
by its values on
$\Omega_m$, the sector of $\Rd$ defined by
\begin{equation}
 \label{dmn}
\Omega_m = \left\{(x_1,\; x_2,\; \cdots,\; x_N) \in
\Rd;\; x_1>0,\; x_2>0,\cdots,\; x_m>0\right\}.
\end{equation}
Note that by definition, an antisymmetric function must take the value $0$
on the boundary $\partial\Omega_m$.  Since the operators $T_i$ defined above commute
with the operations in equation \eqref{nheq}, the study of antisymmetric solutions to \eqref{nheq}
 reduces to the study of solutions on $\Omega_m$ with Dirichlet boundary conditions.
This point is discussed in detail in Section 3 of \cite{TW2}, and that discussion applies as well
to the heat equation with  absorption.  Moreover, as in \cite{TW2}, we will construct certain
classes of antisymmetric solutions to \eqref{nheq} on $\Rd$ by constructing solutions on
$\Omega_m$ and extending them to $\Rd$ by antisymmetry.

Since both the present paper and \cite{TW2} are based on the framework developed in \cite{MTW},
we need to recall
 some definitions and notation used in \cite{MTW}. Let $\rho_m$ be the weight function defined on $\Omega_m$ by
\[
   \rho_m (x) =\dfrac{|x|^{\gamma +2m}}{x_1 \cdots x_m},\ \ \text{for all}
   \ x \in \Omega_m ,
\]
where $0<\gamma<N$.
We consider the Banach space
\begin{equation}
\label{spc}%
\X = \left\{ \psi :\Omega_m \to \R,\quad
\rho_m \psi \in L^\infty(\Omega_m)\right\} ,
\end{equation}
 endowed with the norm
\[
  \|\psi\|_{\X}=\left\|\rho_m \psi \right\|_{ L^\infty(\Omega_m)},
\]
for all $\psi \in \X$.
The closed ball of radius $M$ on $\X$ is denoted by
\begin{equation}\label{ball}
\B= \{ \psi \in \X \ \text{such\ that}\ \|\psi\|_{\X} \le M\}.
\end{equation}
As observed in \cite[p. 344]{MTW}, $\Bs$ the closed ball $\B$ endowed with the
weak$^\star$ topology of $\X$, is a compact metric space (hence complete and separable).

Let $\sigma> 0$. For each $\lambda> 0$, we let $D_{\lambda}^{\sigma}$ denote the dilation operator defined  by
\begin{equation}
\label{dilop} %
 D_{\lambda}^{\sigma} u(x) = \lambda^\sigma u(\lambda x),
\end{equation}
where $u$ is a function defined on $\Omega_m$, or on $\Rd$.
A function $\psi: \Omega_m \to\R$ is homogeneous of degree $-\sigma$  if  $D_{\lambda}^{\sigma}\psi=\psi$ for all $\lambda > 0$.
The operators $D_{\lambda}^{\sigma}, \lambda >0$, act on the spaces $\X$, but leave the norm invariant, {i.e.} leave  the ball $\B$ invariant, if and only if $\sigma=\gamma+m$. In fact, we have
\begin{equation}
\label{dilisom} %
 \|D_{\lambda}^{\sigma}\psi\|_{\X}={\lambda}^{\sigma}   \|\rho_m\psi(\lambda \cdot)\|_{L^\infty(\Omega_m)}=
 {\lambda}^{\sigma - (\gamma + m)}  \|\rho_m(\lambda \cdot)\psi(\lambda \cdot)\|_{L^\infty(\Omega_m)} =
  {\lambda}^{\sigma - (\gamma + m)}\|\psi\|_{\X},
\end{equation}
for all $\lambda > 0$ and $\psi \in \X$.

The function $\psi_0$ defined on $\Omega_m$ by
\begin{equation}
\label{psi0} %
\psi_0(x)= c_{m,\gamma} \left(\rho_m(x)\right)^{-1} = c_{m,\gamma} x_1 \cdots x_m |x|^{-\gamma-2m},
 \quad x \in \Omega_m,
\end{equation}
where $c_{m,\gamma}  = \gamma (\gamma+2) \cdots (\gamma+2m-2)$, will play a central role.
It is homogeneous of degree $-(\gamma+m)$,
belongs to $\X$ and satisfies $\|\psi_0\|_{\X} = c_{m,\gamma}$. Moreover,
\begin{equation}  \label{psi0prop} %
D_{\lambda}^{\sigma}\psi_0(x) = \lambda^{\sigma-(\gamma+m)} \psi_0(x) ,
\end{equation}
for all $\sigma, \lambda> 0$, and so $\| D_{\lambda}^{\sigma}\psi_0 \|_{\X}= \lambda^{\sigma-(\gamma+m)} c_{m,\gamma}$.
Its interest lies in the fact that
\begin{equation}
\label{psi00} %
  \psi_0(x)= (-1)^m \partial_1 \partial_2 \cdots \partial_m\left( |x|^{-\gamma}\right),  \quad x \in \Omega_m.
\end{equation}

%Thus, in the case $m=0$, we take $\psi_0(x)=|x|^{-\gamma}$, for $x\neq0$, and we define the space
%\begin{equation} \label{spwhole}%
%\mathcal{W}^{\gamma}=\mathcal{X}_{0,\gamma} :=\left\{ u \in L^1_{loc}(\Rz);\ |x|^{\gamma}u(x) \in L^\infty(\Rd) \right\}.
%\end{equation}

The heat semigroup on $\Omega_m$, denoted $e^{t\Delta_m}$, is given by
\begin{equation}
\label{sgomg}%
    e^{t\Delta_m} \psi (x)=\int_{\Omega_m} K_t(x,y) \psi(y) dy\; ,
\end{equation}
for all $t>0$, where
\begin{equation} \label{hksec}
K_t (x, y) = (4\pi t)^{-\frac{N}{2}} \displaystyle \prod_{j=m+1}^{N}
e^{-\frac{|x_j-y_j|^2}{4t}} \prod_{i=1}^{m}
\left[e^{-\frac{|x_i-y_i|^2}{4t}}-e^{-\frac{|x_i+y_i|^2}{4t}} \right] .
\end{equation}
See, for example, \cite[Proposition 3.1, p. 514]{TW}.
It is well-known that $e^{t\Delta_m}$ is a $C_0$ semigroup on $C_0(\Omega_m)$, the space of continuous functions $f:\overline{\Omega}_m\to \R$ such that $f\equiv 0$ on the boundary $\partial \Omega_m$ and $f(x) \to 0$ as $|x|\to \infty$ in $\Omega_m$.
It is also well defined on $\X$ and $e^{t\Delta_m}: \X \to C_0(\Omega_m)\cap\X$ is continuous, for all $t>0$. See \cite[Theorem 1.1, p. 343]{MTW}.
We recall the commutation relation between $e^{t\Delta_m}$ and the operators $D_{\lambda}^{\sigma}$,
\begin{equation}
\label{dilsmg}%
 D_{\lambda}^{\sigma}  e^{\lambda^2 t \Delta_m} = e^{t \Delta_m} D_{\lambda}^{\sigma}
\end{equation}
for all $\lambda>0$ and $\sigma>0$, and for future use we note the following identity, which is immediate to verify,
\begin{equation} \label{hksgident}
 \int_{ \Omega_m} K_t(x,y)\; y_1\cdots y_m\; dy =  x_1 \cdots x_m ,
\end{equation}
for all $t>0$ and all $x\in\Omega_m$.

In terms of behavior on the sectors $\Omega_m$, our goal is to study the well-posedness of the equation \eqref{nheq} on the space  $\X$ and to obtain results on the large time behavior of solutions in the three cases $\alpha=2/(\gamma+m)$,
$\alpha > 2/(\gamma+m)$ and  $\alpha<2/(\gamma+m)$. By interpreting these results for antisymmetric solutions
on $\Rd$, we will extend some know results, \cite[Theorem 1.3, Theorem 1.4]{CDW} and \cite{GVL}, in the case $m=0.$
We now describe these results in detail.

In Section \ref{existsec}, we consider the Cauchy problem
\begin{equation} \label{cpe} %
\left\lbrace
\begin{array}{l}
u_t - \Delta u +  |u|^\alpha u =0,\\
u(0)=u_0 \in \X  .
\end{array}
\right.
\end{equation}
 It is well known that, given any $u_0\in C_0(\Rd)$ there exists a unique function $u\in C( [0,\infty), C_0(\Rd) )$ which is a classical solution of \eqref{nheq} on $\Rd$ for $t > 0$ and such that $u(0)=u_0$,
 which we denote by
 \begin{equation}
\label{sflowrn}%
 u(t)= \mathcal S(t) u_0 ,
\end{equation}
 where $u(t) = u(t,\cdot)$.
Likewise, for any $u_0\in C_0(\Omega_m)$, there exists a unique function $u\in C( [0,\infty), C_0(\Omega_m) )$ which is a classical solution of \eqref{nheq} for $t > 0$ and such that $u(0)=u_0$.  This defines a global semi-flow $\mathcal S_m(t)$ on $C_0(\Omega_m)$.  In other words,
\begin{equation} \label{sflow}%
\mathcal S_m(t) u_0 = u(t) ,
\end{equation}
where $u(t) = u(t,\cdot)$ is the solution of \eqref{nheq} with initial value
$u_0 \in C_0(\Omega_m)$.  In fact, existence and uniqueness of solutions in $C_0(\Omega_m)$ follows from
the existence and uniqueness of solutions in $u_0\in C_0(\Rd)$ since
$\mathcal  S(t)$ preserves antisymmetry: it suffices to consider the anti-symmetric extension of $u_0\in C_0(\Omega_m)$
to an element of  $C_0(\Rd)  \cap  \mathcal{A}$.

Similarly, given any $u_0\in L^q(\Omega_m)$, $1\le q< \infty$, we deduce by Kato's parabolic inequality (see Lemma \ref{kato} and Corollary \ref{cki} in the appendix) and the fact that $\D$ is dense in $L^q(\Omega_m)$, that there exists a unique
$u\in C( [0,\infty), L^q(\Omega_m) )$ which is a classical solution of \eqref{nheq} for $t > 0$ and such that $u(0)=u_0$. Alternatively, see \cite[Proposition 1.1, p. 261]{GVL} for a proof using accretive operators. Again by preservation of antisymmetry, the result of \cite{GVL}, valid for $\R^N$,  holds also on $\Omega_m$. Thus, the semi-flow $\mathcal S_m(t)$ extends to
$L^q(\Omega_m)$ and formula \eqref{sflow} is valid also
 for $u_0 \in L^q(\Omega_m)$.

Here we consider initial data $u_0\in \X.$ Our first main result is the following.
%%%%%%%%%%%%%%%%%%%%%%%
\begin{theorem} \label{exist}
Let $m\in \{1,2,\cdots,N\}$, $0<\gamma<N$ and $\alpha>0$. If $u_0\in \X$, then there exists a unique solution $u\in C( (0,\infty), C_0(\Omega_m) )$ of the equation \eqref{nheq} such that
\begin{enumerate}
\item [(i)] $u(t) \to u_0$ in $L^1_{\text{loc}}(\Omega_m)$ as $t\to 0$ ;
\item [(ii)] there exists $C > 0$, independent of  $u_0$, such that $\|u(t)\|_{\X} \le C\|u_0\|_{\X}$ for all $t > 0$.
\end{enumerate}
In addition,  the following properties hold.
\begin{enumerate}
\item [(iii)] For all $v_0 \in \X$, $|u(t) - v(t)|\le e^{t \Delta_m} |u_0 - v_0|$, where $v$ is the solution of \eqref{nheq} with initial value $v_0$ satisfying (i) and (ii).
\item [(iv)] There exists $C > 0$ such that $|u(t,x)| \le C \; x_1\cdots x_m \left(t+|x|^2 \right)^{-\frac{\gamma+2m}{2}} \|u_0\|_{\X}$ for all $t > 0$ and for all $u_0\in \X$.
\item [(v)] The solution $u(t)$ satisfies the integral equation
\begin{equation}
\label{nheqINTE}
          u(t)= e^{t\Delta_m} u_0 - \int_0^t e^{(t-s)\Delta_m}\left( |u(s)|^{\alpha} u(s)  \right) ds,
\end{equation}
for all $t>0$, where the integrand is in $L^1((0,t) ; C_0(\Omega_m))$.
\item [(vi)] If $v_0 \in \X$, $v_0 \ge 0$, and $|u_0| \le v_0$, then $|u(t)| \le  v(t)$, where $v$ is the solution of \eqref{nheq} with initial value $v_0$ satisfying (i) and (ii).
\end{enumerate}
\end{theorem}
%%%%%%%%%%%%%%%%%%%%%%%
In other words, the nonlinear operators $\mathcal S_m(t)$, $t>0$, extend in a natural  way to $\X$. We remark that in the case $\alpha<2/(\gamma+m),$ this well-posedness result  was  established in \cite[Theorems 2.3 and 2.6]{TW2} by a different method and with plus and minus sign in the term of the nonlinearity. Furthermore, the  analogous results on the whole space $\Rd$ follows from \cite[Theorem 8.8, p. 536]{CDEW}.

\begin{definition} \label{sf}
Let $m\in \{1,2,\cdots,N\}$, $0<\gamma<N$ and $\alpha>0$.
Given $u_0\in \X$ we set
 \[
          \mathcal S_m(t) u_0 = u(t),
\]
for all $t>0$, where $u\in C( (0,\infty), C_0(\Omega_m) )$ is the unique solution of \eqref{nheq} satisfying $(i)$ and $(ii)$ of the Theorem \ref{exist}.
\end{definition}

We also establish the continuous dependence properties of solutions of equation \eqref{nheq} with initial values in $\X$.
%%%%%%%%%%%%%%%%%%
\begin{theorem} \label{depcont}
Let $m\in \{1,2,\cdots,N\}$, $0<\gamma<N$ and $M>0$. It follows that  $\mathcal S_m(t)$ is continuous $\Bs \to C_0(\Omega_m)$, for all $t>0$, where $\Bs$ denotes the compact metric space  topology induced by the weak* topology on $\B$.
\end{theorem}

%%%%%%%%%%%%%%%%%%
It is well-known that any solution $u(t)$ of \eqref{nheq}, for example as constructed in Theorem~\ref{exist},  is always bounded by the spatially independent solution, more precisely
\begin{equation}
\label{upperbd}
|u(t,x)| \le \left(\frac{1}{\alpha t}\right)^{\frac{1}{\alpha}}
\end{equation}
for all $t > 0$, throughout the spatial domain of existence.  See for example \cite[page 261]{GVL}.  In addition,
it is clear from Theorem~\ref{exist} that if $u$ is the  solution of  \eqref{nheq} with positive initial data $u_0\geq 0$ then
\begin{equation}
\label{uleqlinear}
u(t)\leq e^{t\Delta_m} u_0,
\end{equation}
for any $t>0.$ We have the following upper estimate for solutions of \eqref{nheq} which  combines \eqref{uleqlinear} and \eqref{upperbd} into one estimate which implies them both.
Its proof is given in Section~\ref{upperbound}.

\begin{proposition}
\label{preciseupperestimate}
Let $N\geq 1,\; m\in \{1,\cdots, N\},\; 0<\gamma<N$ and $\alpha>0.$ Let $u_0 \in \X$, $u_0 \ge 0.$ Then the solution $u$ of \eqref{nheq} with initial data $u(0)=u_0$ satisfies the following upper estimate
\begin{equation}
\label{upperestimates3}
u(t,x)\leq \frac{e^{t\Delta_m} u_0(x)}{\left(1 + \alpha t\left(e^{t\Delta_m} u_0(x)\right)^{\alpha}\right)^{\frac{1}{\alpha}}}
\end{equation}
for all $t>0,$ and all $x\in \Omega_m.$
\end{proposition}

%%%%%%%%%%%%%%%%%%

 After proving global well-posedness of the Cauchy problem \eqref{cpe}, {\it i.e.} Theorems~\ref{exist} and \ref{depcont}, we seek to describe the large time behavior of solutions of \eqref{nheq} on $\Omega_m$  with initial values in $\X$.
 Our basic approach is to study the effect of certain space-time dilations on such a solution, and to relate the resulting
behavior to the effect of related spatial dilations on the initial value.  In particular
we consider the space-time dilation operators $\Gamma_\lambda^{\sigma}, \; \lambda > 0,$ defined by
\begin{equation}
\label{dilsolnew}%
\Gamma_\lambda^\sigma u(t, x) = \lambda^{ \sigma } u( \lambda^2 t, \lambda x)
=  D_\lambda^{\sigma} [u(\lambda^2 t)](x),
\end{equation}
for all $\lambda,\sigma > 0$. If $u\in C((0,\infty), C_0(\Omega_m))$ is solution of the equation \eqref{nheq} then $\Gamma_\lambda^\sigma u$ is solution of \eqref{nheq} if and only if $\sigma=2/\alpha$. Moreover, if a solution $u$ has initial value $u_0$, either in the sense of $C_0(\Omega_m)$ or in some more general sense, then $\Gamma_\lambda^{2/\alpha} u$ has initial value $D_\lambda^{2/\alpha} u_0$. If $u_0\in\X$, the function $D_\lambda^{2/\alpha}u_0$ belongs to $\X$, for all $\lambda>0,$ and the uniqueness of solutions of \eqref{nheq} implies that $\Gamma_\lambda^{2/\alpha} u$ coincides with
$\S_m(\cdot) D_\lambda^{2/\alpha}u_0$.
Thus, we have the following relation
\begin{equation} \label{invpropsmgtime}%
\Gamma_{\lambda}^{2/\alpha} \left[ \S_m(\cdot)u_0 \right] = \S_m(\cdot) \left[ D_{\lambda}^{2/\alpha} u_0 \right],
\end{equation}
for all $u_0\in \X$.
We emphasize that at this point there is no assumed relationship between $\alpha$ and $m$.
Formula \eqref{invpropsmgtime} holds for any semiflow generated by \eqref{nheq} in  place of $\S_m(\cdot)$,
as long as the space of initial values is invariant under the dilations $D_{\lambda}^{2/\alpha}$ and initial
values give rise to unique solutions.

A solution $u$ of \eqref{nheq} is self-similar if $\Gamma_\lambda^{2/\alpha} u =u$, for all $\lambda>0$, or equivalently if
\begin{equation} \label{selfsol}%
 u(t, x) = t^{- \frac{1}{\alpha} } f(x/\sqrt{t}) = D_{\frac{1}{\sqrt t}}^{2/\alpha} f( x) ,
\end{equation}
where $f(x)=u(1,x)$ is called the profile of $u$. It follows that if a self-similar solution $u$ of \eqref{nheq} has initial value $u_0$, then $D_\lambda^{2/\alpha} u_0 = u_0$, for all $\lambda>0$, i.e. $u_0$ is homogeneous of degree $-2/\alpha$.
Conversely, if $u_0$ is homogeneous of degree $-2/\alpha$ and $u(t)$ is a solution with initial value $u_0$ in some appropriate sense, then $\Gamma_\lambda^{2/\alpha} u$ has the same initial value,
for all $\lambda>0$.
Assuming that uniqueness of solutions having a given initial value has been proved in the
appropriate class of functions, one then concludes that $u = \Gamma_\lambda^{2/\alpha} u$, for all $\lambda>0$, i.e. that $u$ is a self-similar solution.

More generally, we say that a solution $u$ of \eqref{nheq} is asymptotically self-similar if
\begin{equation} \label{aselfsol}%
\lim_{\lambda\to\infty} \Gamma_\lambda^{2/\alpha} u = U,
\end{equation}
in some appropriate sense, and that $U$ is also a solution to \eqref{nheq}. If so, the limit is necessarily a
self-similar solution. See Section 3 of \cite{CDWMult} for a discussion of several equivalent definitions of asymptotically self-similar solutions.
Formally, if we put $t=0$ in \eqref{aselfsol}, we obtain that
\begin{equation} \label{initcv}%
       \lim_{\lambda\to\infty} D_\lambda^{2/\alpha} u_0 = \varphi,
\end{equation}
where $\varphi=U(0)$ is homogeneous of degree $-2/\alpha$.
In the Section  \ref{criticalsec} we study the long time asymptotic behavior of solutions to
\eqref{nheq} with initial values in $\X$ in the case $\alpha = 2/(\gamma + m)$.  The first result shows that
\eqref{initcv} implies \eqref{aselfsol}.

%%%%%%%%%%%%
\begin{theorem}
\label{inihomasysol}
Let   $m\in \{1,2,\cdots,N\}$, $0<\gamma<N$ and $\psi \in \B$. Let $\alpha>0$ be such that
\[
          \alpha=\frac{2}{\gamma+m}.
\]
Suppose that there exists $\varphi \in \B$ such that
$\displaystyle \lim_{\lambda \to \infty} D_\lambda^{\gamma+m} \psi =
\varphi$ in $\Bs$. It follows that $\varphi$ is homogeneous of
degree $-(\gamma+m)$ and that the solution $u(t)=\S_m(t)
\psi$ is asymptotically self-similar to the
self-similar solution $U(t)=\S_m(t) \varphi$.
\end{theorem}

As is by now well established \cite{CDWL,CDW,CDWMult}, the notion of asymptotically self-similar solution can be naturally
extended by allowing different limits in
 \eqref{initcv} and \eqref{aselfsol} along
different sequences $(\lambda_n)_{n\ge 0}$, with $\lambda_n \to \infty$.
The next step in our analysis it to generalize Theorem~\ref{inihomasysol} in this fashion.
To accomplish this, for $u_0 \in \X$ and $M\ge \|u_0\|_{\X}$, we consider the set of all accumulation points of
$D_{\lambda}^{\gamma+m} u_0 $, as $\lambda \to \infty$, given by
\begin{equation}
\label{Omg} %
 \mathcal{Z}^\gamma (u_0) = \left\{z \in \B;\ \exists\ \lambda_n \to \infty \ \text{such that}\
         \lim_{n\to \infty} D_{\lambda_n}^{\gamma+m} u_0 = z \ \text{in}\
         \Bs \right\} .
\end{equation}
Since $\Bs$ is a compact metric
space,  $ \mathcal{Z}^{\gamma}(u_0)$ is nonempty compact subset, for all $u_0\in \X$, and independent of $M\ge \|u_0\|_{\X}$ by \cite[Proposition 3.1, p. 356]{MTW}. In particular, if $u_0$ is homogeneous of degree $-(\gamma+m)$,
then $ \mathcal{Z}^{\gamma}(u_0)=\{ u_0\}$.
We set $u(t)=\S_m(t) u_0$ and we also define the omega-limit set of all accumulation points of
$\Gamma_{\sqrt t}^{\gamma+m} u(1, \cdot)=t^{\frac{\gamma+m}{2}} u(t,\sqrt t \; \cdot)$, as $t\to \infty$, by
\begin{equation} \label{omg}%
 \mathcal{Q}^\gamma(u_0) = \left\{f \in C_0(\Omega_m); \exists t_n \to \infty \
\text{such that} \lim_{n\to \infty}\|\Gamma_{\sqrt t_n}^{\gamma+m}S_m(1)u_0 - f \|_{L^\infty(\Omega_m)}=0
\right\} .
\end{equation}
The relation \eqref{invpropsmgtime} and Theorem \ref{depcont} are the essential elements needed to investigate the relationship between $\mathcal{Q}^{\gamma}(u_0)$ and $\mathcal{Z}^{\gamma}(u_0)$, which is given by our
next main result.

%%%%%%%%%%%%%%%%
\begin{theorem}
\label{relomg}
Let $m\in \{1,2,\cdots,N\}$,  $0 <\gamma< N$ and let $\alpha>0$ be such that
\[
          \alpha=\frac{2}{\gamma+m}.
\]
 If $u_0\in \X$, then
\[
        \mathcal{Q}^\gamma (u_0) = \mathcal S_m(1)  \mathcal{Z}^\gamma (u_0) .
\]
In particular, $\mathcal{Q}^\gamma (u_0) \subset \mathcal S_m(1)\Bs$ and is therefore
a compact subset of $C_0(\Omega_m)$.
\end{theorem}

The last relation shows that in the case $\alpha = 2/(\gamma + m)$ the complexity in the large time behavior of a solution, as expressed in $\mathcal{Q}^\gamma(u_0)$, is determined by the complexity in the spatial asymptotic behavior of its initial value as expressed in $\mathcal{Z}^\gamma (u_0)$. Furthermore, Theorem \ref{relomg} above is inspired from \cite[Theorem 1.3, p. 83]{CDW} which requires $\alpha\ge2/N$, and we observe that in Theorem \ref{relomg}, if $\gamma+m>N$, then $\alpha<2/N$.  Since $\Bs$ is separable and $\mathcal{Z}^\gamma (u_0)$ can contain any countable subset of $\Bs$, we show that $\mathcal{Z}^\gamma (U_0)=\Bs$ for some choice of $U_0\in\Bs$.

Using \cite[Theorem 1.4, p. 345]{MTW}, we obtain the following result.

\begin{corollary} \label{univsol}
Let $m\in \{1,2,\cdots,N\}$, $0 <\gamma< N$ and  $M>0$.
Let $\alpha>0$ be such that
\[
          \alpha=\frac{2}{\gamma+m}.
\]
 Then, there exists
\[
U_0 \in \B \cap C^\infty(\Omega_m) \cap C_0(\Omega_m)
\]
such that
\[
        \mathcal{Q}^\gamma(U_0) = \mathcal S_m(1) \B .
\]
\end{corollary}

\begin{remark}\label{clslin}
{\rm If $u_0$ belongs to $\X \cap\mathcal{X}_{m,\gamma'}$ with $\gamma< \gamma'< N$, then
$\mathcal{Z}^\gamma (u_0) =\{0\}$.
 In fact, for all $\lambda>0$,
\[
     \left| D_{\lambda}^{\gamma+m} u_0 (x) \right| = \lambda^{\gamma+m} |u_0(\lambda x) | \le C \lambda^{\gamma-\gamma'} |x|^{-(\gamma'+m)} \to 0,
\]
as $\lambda\to \infty$ uniformly on $\{x\in\Omega_m;\; |x|\ge \varepsilon\}$, for all $\varepsilon>0$. Thus, $ \mathcal{Z}^{\gamma}(u_0) =\{0\}$. For example, if $\alpha=2/(\gamma+m)>2/(\gamma'+m)$, the function $\varphi(x)=x_1\cdots x_m |x|^{-\gamma'-2m}1\!\!1_{\{|x|>1\}} \in \mathcal{X}_{m,\gamma'}\cap\mathcal{X}_{m,\gamma} $. It follows, from Theorem \ref{relomg}, that $\mathcal{Q}^{\gamma} (\varphi)=\{0\}$. However, we might have $\mathcal{Q}^{\gamma'} (\varphi)\not=\{0\}.$ }
\end{remark}

%According to the above remark, it is natural to describe   $\mathcal{Q}^\gamma (\phi) \neq\{0\}$ for $\alpha>2/(\gamma+m).$

In Section \ref{linearsec} of this paper, we consider the case $\alpha>2/(\gamma+m)$.
Since $\alpha \neq 2/(\gamma+m)$ there is a disconnect between the transformations which
preserve the set of solutions to \eqref{nheq}, {\it i.e.} $\Gamma_{\lambda}^{2/\alpha}$,
and those which leave invariant the norm of the space $\X$ where the solutions live,
{\it i.e.} $\Gamma_{\lambda}^{\gamma + m}$.  Indeed, by \eqref{dilsolnew} and \eqref{dilisom}
it follows that for $u_0 \in \X$

\begin{equation}
\label{alphaneqgammaplusm}
\|\Gamma_{\lambda}^{\sigma}\S_m(t)u_0\|_{\X} = \|D_{\lambda}^{\sigma}\S_m(\lambda^2 t)u_0\|_{\X}
= \lambda^{\sigma - (\gamma + m)}\|\S_m(\lambda^2 t)u_0\|_{\X}.
\end{equation}
Since $\|\S_m(\lambda^2 t)u_0\|_{\X} \le C\|u_0\|_{\X}$, for some $C > 0$, by Theorem~\ref{exist}, it follows,
setting  $\sigma = 2/\alpha$ in \eqref{alphaneqgammaplusm}, that
if $2/\alpha < \gamma + m$, then $\|\Gamma_{\lambda}^{2/\alpha}\S_m(t)u_0\|_{\X} \to 0$
as $\lambda \to \infty$,
uniformly for all $u_0$  in a bounded set of $\X$ and all $t > 0$.

It is clear from \eqref{alphaneqgammaplusm} that for $u_0 \in \X$, the transformations most
likely to yield some nontrivial asymptotic behavior are $\Gamma_{\lambda}^{\gamma + m}$.
In other words, we still need to study  $\mathcal{Q}^\gamma(u_0)$ as given by \eqref{omg},
and likewise $ \mathcal{Z}^\gamma (u_0)$ as given by \eqref{Omg}.
However, we cannot expect the relationship between these two objects to be given as in
Theorem~\ref{relomg} since the transformations do not preserve solutions of \eqref{nheq}.

If $u$ is a solution of \eqref{nheq} then
$v = \Gamma_{\lambda}^{\gamma + m}u$ is the solution of the equation
\begin{equation}
\label{nheqlam}%
v_t - \Delta v +  \lambda^{2 - (\gamma + m)\alpha}|v|^\alpha v =0.
\end{equation}
If $\alpha>2/(\gamma+m)$, it follows that as $\lambda \to \infty$, the function $v$ satisfies an equation
which approaches the {\it linear} heat equation.
Hence, we should not be surprised if in this case
$\mathcal{Q}^\gamma(u_0)$ and $ \mathcal{Z}^\gamma (u_0)$ are related by the linear heat equation.
The next theorem makes this idea precise, both in the asymptotically self-similar case, and the
more general case of arbitrary $u_0 \in \X$.  It is analogous to \cite[Lemma 5.1, p. 110]{CDW}.

\begin{theorem} \label{lincase}
Let $m\in\{1,\cdots,N\}$, $0< \gamma<N$ and $M>0$. Let $\alpha$ be such that
\begin{equation} \label{cndlin}
       \alpha> \frac{2}{\gamma+m}.
\end{equation}
We then have the following conclusions.
 \begin{enumerate}
\item [(i)] If $u_0, \varphi \in \B$ is such that
$\displaystyle \lim_{\lambda \to \infty} D_\lambda^{\gamma+m} u_0 =
\varphi$ in $\Bs$, then $\varphi$ is homogeneous of
degree $-(\gamma+m)$ and  $u(t)=\S_m(t) u_0$ is asymptotically self-similar to $U(t)=e^{t\Delta_m} \varphi$.
\item [(ii)] $\mathcal{Q}^{\gamma}(u_0) = e^{\Delta_m}  \mathcal{Z}^\gamma(u_0)$, for all $u_0\in \X$.
\item [(iii)] There exists $U_0 \in \B \cap C^\infty(\Omega_m) \cap C_0(\Omega_m)$, such that $\mathcal{Q}^{\gamma}(U_0) = e^{\Delta_m}  \B$.
\end{enumerate}
\end{theorem}

In Section \ref{nonlinsec} of this paper, we consider the case $\alpha<2/(\gamma+m)$.
As in the case of Theorem~\ref{lincase}, the transformations which leave solutions invariant, {\it i.e.}
$\Gamma_{\lambda}^{2/\alpha}$, do not leave invariant the norm of  $\X$, which is the space where
the solution lives.  Nonetheless, unlike in the case $\alpha>2/(\gamma+m)$, the transformations
$\Gamma_{\lambda}^{2/\alpha}$ reveal nontrivial asymptotic behavior.  Because of
\eqref{alphaneqgammaplusm}, to study this asymptotic behavior, we need to leave the context
of the space $\X$.

This is best illustrated by the result of Gmira and V\'eron \cite{GVL} in the case
of $\R^N$. If we express the upper bound \eqref{upperbd} in terms more suggestive of the long-time asymptotic behavior of the solution, we see that,
considering only positive solutions,
\begin{equation}
\label{upperbdbis}
(\Gamma^{2/\alpha}_{\sqrt t}u)(1,x) = t^{\frac{1}{\alpha}}u(t,x\sqrt t)\le \left(\frac{1}{\alpha}\right)^{\frac{1}{\alpha}}.
\end{equation}
The main result of \cite{GVL} can be stated as follows. Suppose $\alpha < \frac{2}{N}$. Let $u_0 \in L^q(\R^N)$ for some $1 \le q < \infty$, or $C_0(\R^N)$, with $u_0 \ge 0$, be such that for every $k > 0$, there exists $R_0 > 0$ such that
\begin{equation}
\label{GVhyp}
u_0(x) \ge k|x|^{-2/\alpha}, \; |x| \ge R_0,
\end{equation}
{\it i.e.} $\liminf_{|x| \to \infty}|x|^{2/\alpha}u_0(x) = \infty$.  It follows that if $u(t,x)$ is the resulting
solution of \eqref{nheq}, then
\begin{equation}
\label{GVconcl}
 t^{\frac{1}{\alpha}}u(t,x\sqrt t)\to \left(\frac{1}{\alpha}\right)^{\frac{1}{\alpha}}
\end{equation}
uniformly on compact subsets of $\R^N$.  In light of the upperbound \eqref{upperbdbis}, the result
\eqref{GVconcl} is rather sharp.

In the case of the sector $\Omega_m$, we have the following result, where
$C_0^{b,u}(\Omega_m)$ denotes the space of bounded uniformly continuous
functions on $\Omega_m$ which are zero on $\partial\Omega_m$.

\begin{theorem} \label{nonlincaseCptsector}
Let  $m \in \{1, \cdots, N\},$ $0<\gamma<N$ and $\alpha>0$ be such that
\begin{equation}
\label{cndnonlin3}
         \alpha < \frac{2}{\gamma+m} .
\end{equation}
Let $u_0\in\X$ with $u_0 \ge 0$, and let $u(t)=\mathcal{S}_m(t) u_0$
be the resulting solution of \eqref{nheq} as given by Theorem~\ref{exist}.  Suppose that there exist $R_0 > 0$ and $c_0 > 0$ such that
\begin{equation}
\label{condGV}
    u_0(x) \ge c_0\psi_0(x), \; x \in \Omega_m, |x| \ge R_0,
\end{equation}
where $\psi_0$ is given by \eqref{psi0}.
Then
\begin{equation}
\label{gvCompactsector}
\lim_{t\to\infty} t^{\frac{1}{\alpha}}u(t,x\sqrt t)= g(x),
\end{equation}
uniformly on compact subsets of $\overline\Omega_m$,
where $g \in C_0^{b,u}(\Omega_m)$
 is the profile of the  self-similar  solution of \eqref{nheq} given
by  Proposition~\ref{bigselsim}.
\end{theorem}

\begin{remark}
\label{limitcompactsubsets}
 The condition \eqref{condGV} implies that, for any $c>0,$ $$\lim_{|x|\to \infty,\; x_1\cdots x_m|x|^{-m}\geq c}|x|^{2/\alpha}u_0(x)=\infty,$$ since $2/\alpha > \gamma+m.$
\end{remark}

\begin{remark}
\label{explicitestimateofg} Using \eqref{gbeh} below and \eqref{hksec}, we have that $g$ in \eqref{gvCompactsector} satisfies the explicit bound
$$ \alpha^{-1/\alpha}I_m(1,x)\leq g(x)\leq (\alpha \epsilon)^{-1/\alpha}I_m\big((1-\epsilon),x\big),\;  x\in \Omega_m,$$
for all $0<\varepsilon<1,$
where
$$I_m(\delta,x)=\prod_{i=1}^m\left({1\over \sqrt{\pi}}\int_{-{x_i\over 2\sqrt{\delta}}}^{{x_i\over 2\sqrt{\delta}}}e^{-y^2}dy\right).$$
\end{remark}

 In the Section \ref{extensec} of this paper,  we reinterpret the
results of the previous sections on  the global well-posedness and   the asymptotic behavior of $\S_m(t) u_0$, $u_0 \in \X$,
in the case of antisymmetric functions defined on the whole space $\Rd$. Recall that the heat semigroup on $\Rd$ is given by
\begin{equation} \label{hsgrn}
    e^{t\Delta} \varphi=G_t \star \varphi,
\end{equation}
for all $\varphi \in \Sd$,  where $G_t$ is the Gauss kernel  on $\Rd$,
\begin{equation} \label{hk}
G_t (x) = (4\pi t)^{-\frac{N}{2}} e^{-\frac{|x|^2}{4t}} ,
\end{equation}
for all $t>0$ and $x\in\Rd$.  The heat semigroup $e^{t\Delta}$ was studied in \cite{CDW} on the space
\begin{equation} \label{spccdw}
 \mathcal{W}^{\sigma} =\left\{ u \in L^1_{loc}(\R^N\backslash\{0\});\ |x|^{\sigma}u(x) \in L^\infty(\Rd) \right\} ,
\end{equation}
with $0<\sigma<N$. It was observed in \cite{MTW}, that we can consider the case $N\le\sigma<2N$ for some class of antisymmetric initial values in $\mathcal{W}^{\sigma}$. See \cite[Corollary 1.7, p. 346]{MTW} and the discussion just after.

 If $\psi:\Omega_m \to\R$, we denote by $\widetilde{\psi}$ its
pointwise extension to $\Rd$ which is  antisymmetric with respect to $x_1, x_2,\cdots, x_m$.
If $\psi\in\X$, $\widetilde{\psi}$ has a natural interpretation as an element of
$\Sd$. See  \cite[Definition 1.6, p. 346]{MTW}. We also define the space
\begin{equation} \label{extspc}
\widetilde{\X}=\left\{ \widetilde{\psi} ;\ \psi\in \X\right\} \subset \Sd,
\end{equation}
with the norm $\|\varphi\|_{\widetilde{\X}}=\|\varphi{_{|_{\Omega_m}}}\|_{\X}$, for all $\varphi\in \widetilde{\X}$. We also consider,
\begin{equation} \label{extball}
\widetilde{\B}=\left\{ \widetilde{\psi} ;\; \psi\in \B\right\}.
\end{equation}
We denote by $\widetilde{\Bs}$ the ball $\widetilde{\B}$ endowed with the weak$^\star$ topology.
$\widetilde{\Bs}$ inherits the metric space structure from $\Bs$. In addition, we observe that $\widetilde{\X} \subset \mathcal{W}^{\gamma+m}$ with continuous injection. However the two norms are not equivalent. On the other hand, $\Bts \subset (\Bg)^\star$ where $(\Bg)^\star$ denote the closed ball of radius $M$ on $\Wd$ endowed with the weak$^\star$ topology, but here the metric on $\Bts$ is equivalent to the one it inherits from the metric space $(\Bg)^\star$. See Proposition \ref{cvkinds} below.

 The heat semigroup $e^{t\Delta}$ is well-defined on $\widetilde{\X}$ and
\begin{equation} \label{sgrel}%
\widetilde{e^{t\Delta_m} \psi}=e^{t\Delta} \widetilde{\psi}.
\end{equation}
See \cite[Proposition 5.1, p. 361]{MTW}. The last formula is the key to the study the equation \eqref{nheq} in the space $\widetilde{\X}$. The following result is essentially a reformulation of Theorem~\ref{exist} for antisymmetric functions on $\R^N$.
%%%%%%%%%%%%%%
\begin{theorem} \label{extexist}
Let $m\in \{1,2,\cdots,N\}$, $0<\gamma<N$ and $\alpha>0$. If $v_0\in \widetilde{\X}$,  there exists a unique solution $v\in C( (0,\infty), C_0(\Rd) \cap  \mathcal{A})$ of the equation \eqref{nheq}  such that
\begin{enumerate}
\item [(i)] $v(t) \to v_0$ in $L^1_{loc}(\Rz)$ as $t\to 0$ ;
\item [(ii)] there exists $C > 0$, independent of $v_0$, such that $\|v(t)\|_{\widetilde{\X}} \le C\|v_0\|_{\widetilde{\X}}$ for all $t > 0$.
\end{enumerate}
In addition,  the following properties hold.
\begin{enumerate}
\item [(iii)] For all $w_0 \in \widetilde{\X}$, $|v(t) - w(t)|\le e^{t \Delta} |v_0 - w_0|$ ; where $w$ is the solution of \eqref{nheq} with initial value $w_0$ satisfying (i) and (ii).
\item [(iv)] $v(t)$ satisfies the integral equation
\[
          v(t)= e^{t\Delta} v_0 - \int_0^t e^{(t-s)\Delta}\left( |v(s)|^{\alpha} v(s)  \right) ds,
\]
for all $t>0$.
\end{enumerate}
\end{theorem}

%%%%%%%%%%%%%%%%%%%%%%
Since $\widetilde{\X}\subset \mathcal{W}^{\gamma+m}$, where $\mathcal{W}^{\gamma+m}$ is given by \eqref{spccdw}, the last result gives a new class of initial values for which we have  global well-posedness of solutions in the case $\alpha<2/N$ (when $\gamma+m>N$). See \cite{CDEW} and  \cite[Section 4]{CDW}  for information about  non-uniqueness of solutions in the case $\alpha < 2/N$.

The semiflow $\mathcal S(t)$ defined by \eqref{sflowrn} extends to
$\widetilde{\X}$ as the following.

\begin{definition} \label{extsf}
Let $m\in \{1,2,\cdots,N\}$, $0<\gamma<N$ and $\alpha>0$.
Given $v_0\in \widetilde{\X}$ we set
 \[
          \mathcal S(t) v_0 = v(t),
\]
for all $t>0$, where $v\in C( (0,\infty), C_0(\Rd) \cap  \mathcal{A})$ is the unique solution of \eqref{nheq} given by Theorem \ref{extexist}.
\end{definition}
From the construction in Theorem \ref{extexist} and the uniqueness part we have the following formula
\begin{equation}
\label{relsflows}
              \mathcal S(t)  \widetilde{u_0} =  \widetilde{\mathcal S_m(t) u_0 }
\end{equation}
for all $t>0$ and $u_0\in\X$.
As in the case of the sectors $\Omega_m$, i.e. the flow  $\mathcal S(t)$  depends continuously on the  initial values.  The following is an adaptation of Theorem~\ref{depcont}.

%%%%%%%%%%%%%%
\begin{theorem} \label{extdepcnt}
Let $m\in \{1,2,\cdots,N\}$, $0<\gamma<N$ and $M>0$. Then,
 $\mathcal S(t)$ is continuous $\widetilde{\Bs} \to C_0(\Rd)$, for all $t>0$.
\end{theorem}

We now consider the long-time asymptotic behavior of the solutions described
in Theorem~\ref{extexist}.  In analogy with \eqref{Omg} and \eqref{omg} above, and
using a notation consistent with formulas (1.17) and (1.18) in \cite{CDWL} and \cite[Definition 1.2]{CDW}, we make the following
definitions.
For $v_0\in \widetilde{\X}$ we define the $\omega$-limit set of possible asymptotic forms of $v_0$, by
\begin{equation}
 \Omega^{\gamma+m}({v_0})= \left\{z \in \widetilde{\Bs};\ \exists\ \lambda_n
  \to \infty \ \text{\rm such that}\
         \lim_{n\to \infty} D_{\lambda_n}^{\gamma+m} v_0 = z\ {\rm in}\ \widetilde{\Bs}\right\} ,
\label{Omgcdew}
\end{equation}
and
the $\omega$-limit set of all limits
of  $\Gamma_{\sqrt t}^{\gamma+m}\S(1)v_0$, as $t\to \infty$,
by
\begin{equation}
\omega^{\gamma+m}(v_0) =\{f \in C_0(\Rd); \exists t_n
\to
 \infty\  \text{\rm such that}
         \lim_{n\to \infty} \| \Gamma_{\sqrt t_n}^{\gamma+m}\S(1)v_0 -f \|_{L^\infty(\Rd)}=0\} ,
\label{omgcdew} %
\end{equation}

The following three theorems are reformulations of the results above on the asymptotic
behavior of solutions, adapted from the case of the sectors $\Omega_m$ to the case of
antisymmetric functions on $\R^N$, in the three cases: $\alpha$ equals, is greater than, and
is less than $\frac{2}{\gamma + m}$.

%%%%%%%%%%%%%%%
\begin{theorem} \label{extencritical}
Let $m\in \{1,2,\cdots,N\}$,  $0 <\gamma< N$ and $M>0$. Let $\alpha>0$ be such that
\[
          \alpha=\frac{2}{\gamma+m}.
\]
It follows that
 \begin{enumerate}
\item [(i)] if $v_0, \varphi\in \Bt$ are such that
$\displaystyle \lim_{\lambda \to \infty} D_\lambda^{\gamma+m} v_0 =
\varphi$ in $\Bts$, then $\varphi$ is homogeneous of
degree $-(\gamma+m)$ and the solution $v(t)=\S(t)
v_0$ of \eqref{nheq} is asymptotically self-similar to $U(t)=\S(t) \varphi$;
\item [(ii)] $\omega^{\gamma+m}({v_0}) = \S(1) \Omega^{\gamma+m}(v_0)$, for all $v_0\in \widetilde{\X}$;
\item [(iii)] There exists $V_0 \in \Bt \cap C^\infty(\Rd)$, such that $\omega^{\gamma+m}({V_0}) = \S(1) \Bt$.
\end{enumerate}
\end{theorem}

%%%%%%%%%%%
\begin{theorem} \label{extenlincase}
Let $m\in\{1,\cdots,N\}$, $0< \gamma<N$ and $M>0$. Let $\alpha>0$ be such that
\[
       \alpha> \frac{2}{\gamma+m}.
\]
It follows that
 \begin{enumerate}
 \item [(i)] if $v_0, \varphi\in \Bt$ are such that
$\displaystyle \lim_{\lambda \to \infty} D_\lambda^{\gamma+m} v_0 =
\varphi$ in $\Bts$, then $\varphi$ is homogeneous of
degree $-(\gamma+m)$ and  the solution $v(t)=\S(t)v_0$ of \eqref{nheq} is asymptotic to the self-similar solution of the linear heat equation $U(t)=e^{t\Delta} \varphi$;
\item [(ii)] $\omega^{\gamma+m}(v_0) = e^{\Delta}  \Omega^{\gamma+m}(v_0)$, for all $v_0\in \Bt$;
\item [(iii)] there exists $V_0 \in \Bt \cap C^\infty(\Rd)$, such that $\omega^{\gamma+m}({V_0}) = e^{\Delta} \Bt$.
\end{enumerate}
\end{theorem}

%%%%%%%%%%%
\begin{theorem}
\label{nonlincaseRn}
Let  $m \in \{1, \cdots, N\},$ $0<\gamma<N$ and $\alpha>0$ be such that
\begin{equation*}
         \alpha < \frac{2}{\gamma+m} .
\end{equation*}
Let $v_0\in \widetilde\X$ with $v_0|_{\Omega_m} \ge 0$, and let $v(t)=\mathcal{S}(t) v_0$
be the resulting solution of \eqref{nheq} as given by Definition~\ref{extsf}.  Suppose that there exist $R_0 > 0$ and $c_0 > 0$ such that
\begin{equation*}
    v_0(x) \ge c_0\psi_0(x), \; x \in \Omega_m, |x| \ge R_0,
\end{equation*}
where $\psi_0$ is given by \eqref{psi0}.
Then
\begin{equation}
\label{gvlimRn}
\lim_{t\to\infty} t^{\frac{1}{\alpha}}v(t,x\sqrt t)= g(x),
\end{equation}
uniformly on compact subsets of $\R^N$,
where $g \in C^{b,u}(\R^N)$
 is the antisymmetric (bounded, uniformly continuous) profile of the self-similar solution of \eqref{nheq} given
by  Proposition~\ref{newselsim}.
\end{theorem}

Finally, in the appendix, for completeness we give a proof of Kato's parabolic inequality and the main application
for which we use it.  Also, we present some results which we found during the course of research for this article,
which we feel have some independent interest, but which ultimately were not needed for the proofs of the main results.
One of them concerns the lowest eigenvalue and corresponding eigenfunction for $-\Delta$
on $B_1 = \{x \in \Omega_m : |x| < 1\}$ with Dirichlet boundary conditions.

The authors wish to thank Philippe Souplet for several very helpful remarks concerning this research.

%%%%%%%%%%%%%%%%%%%%%%%%%%%%%%%%%
\section{Existence and continuity properties of solutions}
\label{existsec}

The purpose of this section is to study well-posedness of the equation \eqref{nheq} with initial
values in $\X$ and to give the proofs of Theorem \ref{exist} and Theorem \ref{depcont}.
For this purpose, we need several results from \cite{MTW}, sometimes in a slightly stronger version.
The first result below is a slight improvement of \cite[Proposition 2.5, p. 353]{MTW}.

\begin{proposition} \label{cvdd}%
Let $m\in \{1,\cdots, N\}$, $0<\gamma<N$ and $\psi \in \X$. Then,
\[
   e^{t\Delta_m} \psi \longrightarrow \psi,\ \ {\rm as}\  \ t \to 0\ \ {\rm on}\ \ L_{loc}^1(\Omega_m).
\]
In particular, the convergence is also in $\Dd$.
\end{proposition}

\begin{proof}
Let $\psi \in \X$ and $K$ be a fixed compact in $\Omega_m$. Let $\varepsilon=d(0,K)>0$
and $\eta \in
C^\infty(\Rd)$ denote a radial cut-off function, satisfying:
\begin{enumerate}
\item[(i)] $0\leq \eta \leq 1$, for all $x\in \Rd$,
\item[(ii)] $\eta(x)=1$, for all $x\in \Rd$ with $|x|\leq \varepsilon/4$,
\item[(iii)] $\eta(x)=0$, for all $x\in \Rd$ with $|x|\geq \varepsilon/2$.
\end{enumerate}
We write
\begin{equation}\label{decomp}
e^{t\Delta_m} \psi  = e^{t\Delta_m} [\eta\psi] + e^{t\Delta_m} [(1-\eta)\psi] .
\end{equation}
Using the inequality
\begin{equation*}
{\rm e}^{-\frac{(x_{i}-y_{i})^2}{4t}}-{\rm e}^{-\frac{(x_{i}+y_{i})^2}{4t}}
= {\rm e}^{-\frac{x_{i}^2}{4t}} {\rm e}^{-\frac{y_{i}^2}{4t}} \displaystyle
\int_{-\frac{x_{i}y_{i}}{2t}}^{\frac{x_{i}y_{i}}{2t}} {\rm e}^{s} ds
\leq \frac{ x_{i} y_{i} }{t} {\rm e}^{-\frac{(x_{i}-y_{i})^2}{4t}}
\end{equation*}
for all $i\in \{1,\cdots, m\}$, we deduce from \eqref{hksec} that for all $x,y \in \Omega_m$,
\begin{equation}
\label{estbstsglin}
   K_t(x,y)\le t^{-m} \left(\prod_{i=1}^m x_i y_i\right) G_t(x-y).
\end{equation}
Therefore,
\begin{eqnarray*}
|e^{t\Delta_m} [\eta\psi](x)| &\le& C \int_{\Omega_m} K_t(x,y)\; y_1\cdots y_m\; \eta(y)|y|^{-\gamma-2m} dy \\
&\leq & C\; t^{-m}\; x_1 \cdots x_m \int_{|y|\leq \varepsilon /2} G_t(x-y) |y|^{-\gamma} dy.
\end{eqnarray*}
Since, for $x \in K$ (hence $|x|\geq \varepsilon$) and $|y|\leq \varepsilon /2$, we have $|x-y| \geq |x|-|y|
\geq \varepsilon /2$, it follows that
\begin{equation*}
|e^{t\Delta_m} [\eta\psi](x)|  \leq C x_1 \cdots x_m t^{-(m+N/2)} e^{-\frac{\varepsilon^2}{16t}}
 \int_{|y|\leq {\varepsilon /2} } |y|^{-\gamma} dy, \ \forall\; x\in K.
\end{equation*}
This implies that $e^{t\Delta_m} [\eta\psi] \to 0$, a.e. pointwise on $K$, as $t\to 0$. Moreover, by Proposition
\cite[Theorem 1.1 (i), p. 343]{MTW}, we have
\begin{equation*}
|e^{t\Delta_m} [\eta\psi] | \leq C \psi_0, \ \forall\ t>0 .
\end{equation*}
Thus, by the dominated convergence theorem, $e^{t\Delta_m} [\eta\psi] \to 0$ on $L^1(K)$,
as $t \to 0$.

On the other hand, since $(1-\eta)\psi \in L^p(\Omega_m)$ for $p>\max\{1,N/(\gamma+m)\}$, it follows that
$e^{t\Delta_m} [(1-\eta)\psi]\to(1-\eta)\psi$ in $L^p(\Omega_m)$, as $t \to 0$.
In particular, since $K \subset \Omega_m$ is compact, $e^{t\Delta_m} [(1-\eta)\psi]\to \psi $ in $L^1(K)$, as $t\to 0$.
Using \eqref{decomp}, we obtain that $e^{t\Delta_m}\psi \to \psi$  in $L_{loc}^1(\Omega_m)$, as $t\to 0$.
This completes the proof.
\end{proof}

We also need to use a stronger version of \cite[Lemma 2.6, p. 355]{MTW}, as  follows.

\begin{lemma} \label{lemestnew}
Let $m \in \{1, .., N\}$ and $0<\gamma<N$. There exists $C>0$ such that
\begin{equation} \label{estnew}
\vert e^{t \Delta_m} \psi(x) \vert \leq C x_1 \cdots x_m \left(t+|x|^2\right)^{-\frac{\gamma+2m}{2}}\|
\psi\|_{\X},
\end{equation}
for all $t>0$, $x\in \Omega_m$ and  $\psi \in \mathcal{X}_{m,\gamma}$.
\end{lemma}

\begin{proof}
It suffices to prove the Lemma for $\psi=\psi_0$. Since $\psi_0$ is homogeneous, we know that
  $e^{t \Delta_m} \psi_0$ is self similar and so
\begin{equation} \label{slfsmpsi0}%
e^{t \Delta_m}\psi_0(x)=t^{-\frac{\gamma+m}{2}} f\left(\frac{x}{\sqrt{t}}\right),
\end{equation}
where $f:=e^{\Delta_m}\psi_0$.
By \cite[Propostion 2.2, p. 349]{MTW}, we have
\[
    f(x)=e^{\Delta_m} \psi_0(x) \le C\psi_0(x) \le C x_1\cdots x_m |x|^{-\gamma-2m}
\]
for all $x\in\Omega_m$.
 Therefore, there exists $C>0$  such that
\[
 f(x) \leq C  x_1\cdots x_m (1+|x|^2)^{-\frac{\gamma+2m}{2}},
\]
for $|x|\ge1$. On the other hand, for all $x\in \Omega_m$, we have
\[
  f(x) =   \displaystyle
\int_{\Omega_m} K_1(x,y) \; \psi_0(y)\; dy .
\]
Using the inequality \eqref{estbstsglin}, we obtain that
\begin{eqnarray*}
{ f(x) } &\leq & C { x_1\cdots x_m} \displaystyle
\int_{ \Omega_m} G_1(x-y) \; y_1^2\cdots y_m^2
|y|^{-\gamma - 2m} \; dy\\
&\leq& C\;  { x_1\cdots x_m} \displaystyle
\int_{\Rd} {\rm e}^{-\frac{|x-y|^2}{4}}  |y|^{-\gamma} dy \\
&\le& C \; { x_1\cdots x_m}  \left({\rm e}^{\Delta} |\cdot|^{-\gamma}\right)(x) \\
 &\leq& C\; { x_1\cdots x_m} \; (1+|x|^2)^{-\frac{\gamma}{2}},
\end{eqnarray*}
by \cite[Corollary 8.3, p. 531]{CDEW}. Hence, for $|x|\le 1$, we have
\[
 f(x)\leq C \; x_1\cdots x_m\; (1+|x|^2)^{-\frac{\gamma+2m}{2}}.
\]
Therefore, there exists $C>0$  such that
\[
 f(x)\leq C \; x_1\cdots x_m\; (1+|x|^2)^{-\frac{\gamma+2m}{2}},
\]
for all $x\in\Omega_m$. Using the relation \eqref{slfsmpsi0}, we deduce that
\[
 e^{t \Delta_m}\psi_0 (x) \leq C \; x_1\cdots x_m\; \left(t+|x|^2\right)^{-\frac{\gamma+2m}{2}}.
\]
This proves the result.\\\end{proof}
%%%%%%%%%%%%

The following is a version of \cite[Corollary 8.3, p. 531]{CDEW} adapted from $\R^N$ to $\Omega_m$.

\begin{corollary} \label{esttaunew}
Let $m\in\{1,\cdots,N\}$, $0<\gamma<N$ and $A>0$. There exists $C>0$ such that if $\tau\ge0$ and
$u_0 \in \X$ is such that
$|u_0(x)|\le Ax_1\cdots x_m\; (\tau+|x|^2)^{-\frac{\gamma+2m}{2}}$ for $x\in\Omega_m$, then
\[
   \left| e^{t\Delta_m} u_0(x) \right| \le C\; x_1\cdots x_m\; (\tau+t+|x|^2)^{-\frac{\gamma+2m}{2}},
\]
for all $t>0$ and all $x\in\Omega_m$.
\end{corollary}

\begin{proof} From Lemma \ref{lemestnew}, the result is true for $\tau=0$. Next, we consider the case $\tau=1$. We put $g(x)= x_1\cdots x_m\;(1+|x|^2)^{-\frac{\gamma+2m}{2}}$.
 Using \eqref{estbstsglin}, we obtain that
\begin{eqnarray*}
        e^{t\Delta_m} g (x)  &\le \; t^{-m} \; x_1\cdots x_m\;  \displaystyle\int_{ \Rd} G_t(x-y)\;
 \; y_1^2\cdots y_m^2\; \left(1+|y|^2\right)^{-\frac{\gamma+2m}{2}} \; dy \\
&\le  \; t^{-m} \; x_1\cdots x_m\; \displaystyle\int_{ \Rd} G_t(x-y)\;
 \; \left(1+|y|^2\right)^{-\frac{\gamma}{2}} \; dy .
\end{eqnarray*}
By \cite[Corollary 8.3, p. 531]{CDEW}, we have
\begin{eqnarray*}
       e^{t\Delta_m} g(x)  &\le&  C  \; x_1\cdots x_m\; t^{-m} \left(1+t+|x|^2\right)^{-\frac{\gamma}{2}} \\
 &\le&  C \; x_1\cdots x_m\;  \left(1+t+|x|^2\right)^{-\frac{\gamma+2m}{2}} ,
\end{eqnarray*}
for $t\ge 1+|x|^2$, so that  $(2t)^{-m} \le  \left(1+t+|x|^2\right)^{-m}$.

If $t\le 1+|x|^2$, we have $ \left(1+|x|^2\right)^{-\frac{\gamma+2m}{2}} \le C \left(1+t+|x|^2\right)^{-\frac{\gamma+2m}{2}}$, so it suffices to prove that
\begin{eqnarray*}
       e^{t\Delta_m} g(x)  &\le& C  \; x_1\cdots x_m\; \left(1+|x|^2\right)^{-\frac{\gamma+2m}{2}} .
\end{eqnarray*}
Using \eqref{hksgident}, we obtain that
\begin{eqnarray*}
 e^{t\Delta_m} g(x)  =
 \int_{ \Omega_m} K_t(x,y)\; y_1\cdots y_m\;
 \left(1+|y|^2\right)^{-\frac{\gamma+2m}{2}} \; dy \le
  x_1 \cdots x_m .
\end{eqnarray*}
Hence for $|x|\le 1$,
\begin{equation*}
        e^{t\Delta_m} g(x)  \le C  \; x_1 \cdots x_m \;  \left(1+|x|^2\right)^{-\frac{\gamma+2m}{2}} .
\end{equation*}
In addition, $g\in \X$ so by \cite[Theorem 1.1 (i), p. 343]{MTW},
\begin{equation*}
     e^{t\Delta_m} g(x)   \le C \psi_0 (x) \le  C \;x_1\cdots x_m\; |x|^{-\gamma-2m}.
\end{equation*}
Therefore, if $|x|> 1$, so that  $ \left(1+|x|^2\right)^{\frac{\gamma+2m}{2}} \le (2|x|)^{\gamma+2m}$, we have
\begin{equation*}
       e^{t\Delta_m} g(x)  \le C \;x_1\cdots x_m\;  \left(1+|x|^2\right)^{-\frac{\gamma+2m}{2}}.
\end{equation*}
It follows that
\begin{equation} \label{tau1est}
      \left| e^{t\Delta_m} u_0(x) \right| \le  e^{t\Delta_m} g(x) \le C \;x_1\cdots x_m\;
       \left(1+t+|x|^2\right)^{-\frac{\gamma+2m}{2}} ,
\end{equation}
for all $x\in\Omega_m$ and all $t>0$. This proves the result for $\tau =1$.

 For the general case, we proceed by scaling and observe that
\[
            D_{\frac{1}{\sqrt \tau}}^{\gamma+m} g(x)=x_1\cdots x_m\;(\tau+|x|^2)^{-\frac{\gamma+2m}{2}} .
\]
Using formula \eqref{dilsmg} and the inequality \eqref{tau1est}, we obtain
\[
 \left| e^{t\Delta_m} u_0(x) \right| \le e^{t \Delta_m} [D_{\frac{1}{\sqrt \tau}}^{\gamma+m} g](x)=  D_{\frac{1}{\sqrt \tau}}^{\gamma+m}[e^{\frac{t}{\tau} \Delta_m} g](x) \le C \; x_1\cdots x_m\; (\tau+t+|x|^2)^{-\frac{\gamma+2m}{2}}.
\]
This completes the proof.
\end{proof}

We will also use the following lemma, which gives a property of convergence
 in $L^1_{\text{loc}}(\Omega_m)$ which is not shared by convergence in $\Dd$.

\begin{lemma}
 \label{cvlocsg}
  Let $(w_k)_{k\ge 1} \subset \B$ and $w\in\B$ be
such that $\ w_k \underset{k\to\infty}{\longrightarrow} w$ in  $L^1_{\text{loc}}(\Omega_m)$.
Then
\[e^{t\Delta_m} |w_k| \underset{k\to\infty}{\longrightarrow} e^{t\Delta_m} |w| \ \ {\rm in } \ \  C_0(\Omega_m).
\]
\end{lemma}

\begin{proof}
Since $w_k \to w$ in $L^1_{\text{loc}}(\Omega_m)$, then $|w_k| \to |w|$ in $L^1_{\text{loc}}(\Omega_m)$ hence $|w_k| \to |w|$ in $\Dd$.
From \cite[Proposition 3.1 (i), p. 356]{MTW}, and since $(|w_k|)_{k\ge 1}, |w|\subset \B$ we deduce that $|w_k| \to |w|$ in $\Bs$. Since by \cite[Proposition 4.1 (ii), p. 359]{MTW}, $e^{t\Delta_m}: \Bs \to C_0(\Omega_m)$ is continuous, it follows that $e^{t\Delta_m} |w_k| \to e^{t\Delta_m} |w| $ on $C_0(\Omega_m)$, as $k\to\infty$.
\end{proof}
%%%%%%%%%

%%%%%%%%
We now give the proof of Theorem \ref{exist}.

\begin{proof}[\bf Proof of Theorem \ref{exist}]

Let $u_0\in \X$ and let $(\mathcal{K}_n)_{n\geq1}$
be the sequence of nondecreasing compacts in $\Omega_m$ defined by:
\begin{equation*}
\mathcal{K}_n = \left\{x \in \Omega_m,\ \text{such that  } d(x, \partial \Omega_m) \geq
\frac{1}{n}\text{ and }|x| \leq n\right\}.
\end{equation*}
We consider the function
\begin{equation*} \label{cutoff}
u_{0,n} =\xi_n  u_0,
\end{equation*}
where $\xi_n$ is a cut-off function satisfying
           \begin{enumerate}
  \item[(i) ] $\xi_n \in C^{\infty}(\Omega_m)$,
  \item[(ii)] $0 \leq \xi_n \leq 1$, for all $x\in \Omega_m$,
  \item[(iii)] $\xi_n(x)=1$, for all $x\in \mathcal{K}_{n}$,
  \item[(iv)] $\xi_n(x)=0$, for all $x\in \Omega_m \backslash \mathcal{K}_{n+1}$.
           \end{enumerate}
Note that, $u_{0,n}\in \X$, for all $n\ge 1$, and
\begin{itemize}
\item $\|u_{0,n}\|_{\X} \le \|u_{0,n+1}\|_{\X} \le \|u_0\|_{\X}$ ;

\item $u_{0,n} \to u_0$ pointwise and in $L^1_{\text{loc}}(\Omega_m)$ (hence in $\Dd$) as $n \to \infty$ ;

\item for a fixed compact $K$ on $\Omega_m$, then there exist $n_0$ such that $u_{0,n} = u_0$ on $K$, for all $n\geq n_0$.
\end{itemize}

\textbf{Existence:} The proof is motivated by the proof of \cite[Theorem 8.8, p. 536]{CDEW}.
Since ${u_{0,n}} \in L^p(\Omega_m),\; 1\le p <\infty$,  we consider the unique
solution $u_n \in C([0,\infty), L^p(\Omega_m)) \cap C((0, \infty), C_0(\Omega_m))$ of \eqref{nheq} with initial value $u_{0,n}\in L^p(\Omega_m)$. It follows from Kato's parabolic inequality (see Corollary \ref{cki} in the appendix)  that, $\forall n,\ell\in\N^\star$,
\begin{equation}\label{ineqkato}
|u_n(t) - u_\ell(t)| \leq e^{t\Delta_m}|u_{0,n} - u_{0,\ell}|, \ \forall\;  t>0.
\end{equation}
Since $|u_{0,n} - u_{0,\ell}| \leq |u_{0,n} - u_{0}|$, for all $\ell >n$, we have that
\begin{equation}\label{limkato}
|u_n(t) - u_\ell(t)| \leq e^{t\Delta_m}|u_{0,n} - u_{0}|, \
\end{equation}
for all $t>0$ and $\ell >n$. In addition, $\|u_{0,n}-u_{0}\|_{\X}\le 2\|u_0\|_{\X}$, $\forall n\geq1$, and $({u_{0,n}-u_{0}}) \to 0$ in $L^1_{\text{loc}}(\Omega_m)$  as $n\to\infty$, and so it follows from Lemma
\ref{cvlocsg} that $e^{t\Delta_m}|u_{0,n} - u_{0}|\to0$ on $C_0(\Omega_m)$,  as $n\to\infty$. Therefore, from \eqref{limkato}, $u_n(t)$ is a Cauchy sequence in $C_0(\Omega_m)$, for all $t>0$, and so there exists a function
$u(t)$ such that $u_n(t)$ converge to $u(t)$ in $C_0(\Omega_m)$. Furthermore, by letting $\ell \to\infty$ in \eqref{ineqkato}, we obtain that
\[
        |u_n(t) - u(t)| \le e^{(t-\varepsilon)\Delta_m} \left[ e^{\varepsilon\Delta_m}|u_{0,n} - u_{0}| \right]
\]
for all $t>\varepsilon>0$. Since $e^{\varepsilon\Delta_m}|u_{0,n} - u_{0}|\to 0$ in $C_0(\Omega_m)$,  as $n\to\infty$, and $e^{(t-\varepsilon)\Delta_m}$ is $C_0$ contraction on $C_0(\Omega_m)$, we deduce that
$u_n$ converges to $u$ on $L^\infty([\varepsilon, \infty), C_0(\Omega_m))$, for all $\varepsilon>0$.
The limit function $u\in C((0, \infty), C_0(\Omega_m))$ is clearly a solution of \eqref{nheq}.

Again by Corollary \ref{cki}, we have that
\begin{equation*}
      |u_n(t)| \leq e^{t\Delta_m}|u_{0,n}| \le e^{t\Delta_m}|u_{0}| \in \X,
\end{equation*}
for all $n \geq 1$. In addition, by letting $n\to \infty$, we obtain that
\begin{equation} \label{bnds}
      |u(t)| \leq  e^{t\Delta_m}|u_{0}| ,
\end{equation}
and so, by \cite[Theorem 1.1 (i), p. 343]{MTW},  we deduce that
\begin{equation} \label{bdg}
       \|u(t)\|_{\X} \leq  C\|u_0\|_{\X},
\end{equation}
for all $t>0$. This proves (ii).

It remain now to show that $u(t)\to u_0$ on $L^1_{\rm loc}(\Omega_m)$, as $t\to0$. We fix a compact subset $K \subset \Omega_m$ and $n$ such that $u_{0,n} = u_0$ on $K$. Thus,
\[
 \int_{K} |u(t) -u_0| =  \int_{K} |u(t) -u_{0,n}|  \leq  \int_{K} |u(t) -u_n(t)|
+ \int_{K} |u_n(t) -u_{0,n}| .
\]
By letting $\ell \to \infty$ in \eqref{ineqkato}, we have that
\[
      |u_n(t) - u(t)| \leq   e^{t\Delta_m}|u_{0,n} - u_{0}|.
\]
The Proposition \ref{cvdd} shows that $e^{t\Delta_m}|u_{0,n} - u_{0}| \to |u_{0,n} - u_{0}| $ on $L^1_{\rm loc}(\Omega_m)$, as $t\to 0$. Therefore,
\[
    \int_{K} e^{t\Delta_m}|u_{0,n} - u_{0}| \underset{t\to0}{\longrightarrow} \int_{K} |u_{0,n}-u_0| =0,
\]
and so $\int_{K} |u(t) -u_n(t)| \to 0$, as $t\to0$.
Since $u_n \in C([0,\infty), L^p(\Omega_m))$, we have $u_n(t) \to u_{0,n}$ on $ L^p(\Omega_m)$, as $t\to 0$, so that
\[
     \int_{K} |u_n(t) -u_{0,n}| \to 0, \quad {\rm as}\ \ t\to0.
\]
This proves that $u(t)$ converges to $u_0$ on $L^1_{\rm loc}(\Omega_m)$, as $t\to0$, and so (i) is proved.

\textbf{Uniqueness:} Let $s>0$ and $u, v$ two solutions of \eqref{nheq} satisfying (i) and (ii).
We have that
\[
 \left\vert u(t+s) - v(t+s)\right\vert
\leq e^{t \Delta_m} \left\vert u(s) - v(s)\right\vert
\leq e^{t \Delta_m} \left\vert u(s) - u_0\right\vert
+ e^{t \Delta_m} \left\vert v(s) - u_0\right\vert ,
\]
for all $t>s>0$.  Let $M\geq C\|u_0\|_{\X}.$ Since $u(s), v(s)\in \Bs$ for all $s>0$ and $u(s), v(s) \to u_0$ in $L^1_{\rm loc}(\Omega_m)$, as $s\to0$, it follows from the Lemma \ref{cvlocsg} that the right hand side of the last inequality tends to $0$ in $C_0(\Omega_m)$, as $s\to 0$. This gives that $\left\vert u(t+s) - v(t+s)\right\vert \to 0$, as $s\to 0$. But since $u, v \in C((0,\infty), C_0(\Omega_m))$ we deduce that $\left\vert u(t+s) - v(t+s)\right\vert \to \left\vert u(t) - v(t)\right\vert$, as $s\to0$, for every fixed $t>0$. By uniqueness of the limit, we have $u(t)=v(t)$, for all $t>0$.

\textbf{Additional properties:} We next give the proof of the statements (iii), (iv) and (vi).
In fact, by \eqref{bnds}, we have
\[  |u(t)| \le e^{t\Delta_m} |u_0|  ,\]
and so, from Lemma \ref{lemestnew}, we obtain
\[
 |u(t,x)| \le C \; x_1\cdots x_m\;(t+|x|^2)^{-\frac{\gamma+2m}{2}} \|u_0\|_{\X},
\]
for all $t > 0$ and $x\in \Omega_m$.
In addition, if $u_0, v_0\in \X$, we denote $u(t)$ and $v(t)$ the corresponding solutions. For all $n\ge 1$, we let $u_{0,n} = u_0 \xi_n$ and $v_{0,n} = v_{0}
 \xi_n$ where $\xi_n$ is the cut-off function defined by \eqref{cutoff}. Then, for all $n\ge 1$,
\[
     |u_n(t) - v_n(t)| \leq   e^{t\Delta_m}|u_{0,n} - v_{0,n}|.
\]
Letting $n\to\infty$ and using Lemma \ref{cvlocsg}, we deduce that
\[
     |u(t) - v(t)| \leq   e^{t\Delta_m}|u_{0} - v_{0}|.
\]
Finally, assertion (vi) is true since, under the same conditions, $|u_n(t)| \le v_n(t)$, by well-known comparison results.

\textbf{Integral equation:}
Since $u_{0,n} \in L^p(\Omega_m)$, for all $p > \max[1,N\alpha/2]$, the corresponding solution $u_n(t)$ satisfies the integral equation
\[
    u_n(t)= e^{t\Delta_m} u_{0,n} - \int_0^t e^{(t-s)\Delta_m}\left( |u_n(s)|^{\alpha} u_n(s)  \right) ds
\]
for all $t>0$, where each term is in $C([0,\infty);L^p(\Omega_m))$.

%  $u_n(t)$ converge to $u(t)$ on $C_0(\Omega_m)$, as $n\to \infty$.

Since $u_{0,n} \to u_0$ in $\Bs$ as $n\to \infty$, we know, for example by Lemma~\ref{cvlocsg}, that
$e^{t\Delta_m} u_{0,n} \to e^{t\Delta_m} u_{0}$ on $ C_0(\Omega_m)$ as $n\to \infty$, for all $t > 0$.
On the other hand, for all $0 < s < t$,
\[
     e^{(t-s)\Delta_m}\left( |u_n(s)|^{\alpha} u_n(s)  \right)
\to e^{(t-s)\Delta_m}\left( |u(s)|^{\alpha} u(s)  \right)
\]
on $C_0(\Omega_m)$, as $n\to \infty$.
From property (iii) above and \cite[Inequality (1.8), p. 261]{GVL}, we have
\[
     |u_n(s)| \le C\; x_1 \cdots x_m\;(s+|x|^2)^{-\frac{\gamma+2m}{2}}  \quad  {\rm and} \quad |u_n(s)| \le \left(\alpha s\right)^{-1/\alpha},
\]
for all $s> 0$. Therefore, for all $0<\varepsilon< 1$,
\begin{align*}
        |u_n(s)|^{\alpha+1} &= |u_n(s)|^{\alpha(1-\varepsilon)}|u_n(s)|^{1+\alpha\varepsilon}\\
        &\le C|u_n(s)|^{\alpha(1-\varepsilon)}(\omega(x))^{1+\alpha\varepsilon} (s+|x|^2)^{-\frac{(\gamma+m)}{2}(1+\alpha\varepsilon)}\\
& \le \frac{C}{s^{1-\varepsilon}} \omega(x) (s+|x|^2)^{-\frac{(\gamma+m)}{2}(1+\alpha\varepsilon)} .
\end{align*}
where $\omega(x)= x_1 \cdots x_m\;(s+|x|^2)^{-\frac{m}{2}} \le 1$. We then take $\varepsilon < \frac{N-\gamma}{\alpha(\gamma+m)}$ so that $(\gamma+m)(1+\alpha\varepsilon) = \gamma'+m$ with $0<\gamma'<N$, and $0 < \gamma < \gamma'$. It follows that,
\begin{equation} \label{unalpha}
         |u_n(s)|^{\alpha+1} \le\frac{C}{s^{1-\varepsilon}} x_1 \cdots x_m\;(s+|x|^2)^{-\frac{\gamma'+2m}{2}} .
\end{equation}
We deduce by Corollary \ref{esttaunew}, that
\[
              e^{(t-s)\Delta_m} | u_n(s) |^{\alpha+1} \le  \frac{C}{s^{1-\varepsilon}}\; x_1 \cdots x_m\; (t+|x|^2)^{-\frac{\gamma'+2m}{2}} ,
\]
for all $t>s> 0$. Likewise, since $u_n(s) \to u(s)$ in $C_0(\Omega_m)$,
\begin{equation}
\label{integrand}
              e^{(t-s)\Delta_m} | u(s) |^{\alpha+1} \le  \frac{C}{s^{1-\varepsilon}}\; x_1 \cdots x_m\; (t+|x|^2)^{-\frac{\gamma'+2m}{2}} ,
\end{equation}
for all $t>s> 0$, so that $s \to e^{(t-s)\Delta_m} | u(s) |^{\alpha+1}$
is in $L^1((0,t) ; C_0(\Omega_m))$.

We deduce, using the dominated convergence theorem,
\[
    \int_0^t e^{(t-s)\Delta_m}\left( |u_n(s)|^{\alpha} u_n(s)  \right) ds \to  \int_0^t e^{(t-s)\Delta_m}\left( |u(s)|^{\alpha} u(s)  \right) ds ,
\]
as $n\to \infty$ and so that the solution $u(t)$ satisfies:
\[
    u(t)= e^{t\Delta_m} u_{0} - \int_0^t e^{(t-s)\Delta_m}\left( |u(s)|^{\alpha} u(s)  \right) ds.
\]
This proves (v).

Note that \eqref{integrand} implies that
\begin{equation}
\label{intterm}
|\int_0^t e^{(t-s)\Delta_m}\left( |u(s)|^{\alpha} u(s)  \right) ds|
\le Ct^\epsilon x_1x_2 \cdots x_m\; (t+|x|^2)^{-\frac{\gamma'+2m}{2}},
\end{equation}
with $\epsilon$ and $\gamma'$ as above.

\end{proof}

The following lemma is needed to establish Theorem \ref{depcont}.

\begin{lemma}\label{lemex}
Let $(u_n)_{n\ge0}$ be a sequence of solutions of \eqref{nheq}, $u_n\in C((0,\infty), C_0(\Omega_m)),$ satisfying
\begin{equation} \label{bndest}
   | u_n(t,x)| \le C x_1 \cdots x_m\;(t+|x|^2)^{-\frac{\gamma+2m}{2}}, \ \forall\, x\in\Omega_m, \ \forall\, t>0.
\end{equation}
There exists a subsequence $(u_{n_k})_{k\ge0}$ and a solution $g\in C((0,\infty), C_0(\Omega_m))$ of \eqref{nheq} such that
$u_{n_k} \to g$, as $k\to \infty$,  in $C([\tau,\infty),C_0(\Omega_m))$, for every $\tau>0$.
\end{lemma}

\begin{proof}
Fix $\tau >0$. Using \eqref{bndest} with $t=\tau/2$, we deduce that the set $\{u_n(\tau/2), n\geq1\}$ is bounded in $L^p(\Omega_m)$, for all $p$ satisfying $\ {\rm max}\;(1, N/(\gamma+m) )< p \le\infty$. By standard smoothing effects, we see that the $u_n(\tau)=S_m(\tau/2) u_n(\tau/2)$ are uniformly bounded in $W^{1,\infty}(\Omega_m)$. Thus, $\{u_n(\tau)\}$ is relatively compact in $C(\overline{\Omega_m^R})$ for all $R>1$, where $\Omega_m^R=\{x\in \Omega_m; \; |x|\le R\}$. Using the decay estimate \eqref{bndest}, $\{u_n(\tau)\}$ is also relatively compact on $C_0(\Omega_m)$.
By continuous dependence  in $C_0(\Omega_m)$ of \eqref{nheq} it follows that $\{u_n(\cdot), n\geq1\}$ is relatively compact in $C([\tau,T],C_0(\Omega_m))$, for all $T>\tau$,
the limit points being solutions of \eqref{nheq}. Since, by \eqref{bndest}, $\| u_n(t) \|_{L^\infty} \to 0$ as $t\to \infty$, uniformly in $n\ge 1$, we may let $T=\infty$ in the previous property.
By letting $\tau \to 0$ and using a diagonal procedure, we see that there exists a solution $g\in C( (0,\infty),C_0(\Omega_m)) $ of \eqref{nheq} and a subsequence $(u_{n_k})_{k\geq0}$ such that $u_{n_k}   \to g$, as $k\to \infty$, in $C([\tau,\infty),C_0(\Omega_m))$, for every $\tau>0$.
\end{proof}

\begin{proof}[\bf Proof of Theorem \ref{depcont}]
Let $(u_{0,n})_{n\ge0} \subset \B$ and $u_0 \in \B$ such that $u_{0,n} \to u_0$ on $\Bs$, as $n\to \infty$.  Let $u(t) = \S_m(t) u_0$ and $u_n(t) = \S_m(t) u_{0,n}$, for all $t>0$ be the corresponding solutions of \eqref{nheq} as in Definition~\ref{sf}.  By  Theorem \ref{exist} (iv), we have
\begin{equation} \label{estrmp}
   |u_n(t,x)| \le C x_1 \cdots x_m\;(t+|x|^2)^{-\frac{\gamma+2m}{2}},
\end{equation}
for all  $x\in \Omega_m$ and $t>0$.
It follows from the Lemma \ref{lemex} that  there exists a solution $g\in C( (0,\infty),C_0(\Omega_m)) $ of \eqref{nheq} and a subsequence $(u_{n_k})_{k\geq0}$ such that $u_{n_k}   \to g$, as $k\to \infty$, in $C([\tau,\infty),C_0(\Omega_m))$, for every $\tau>0$. To see that $g(t) \to u_0$ in $L^1_{loc}(\Omega_m)$ as $t\to0$, we consider a compact $K\subset \Omega_m$ and let
$\mathcal{O}$  be an open, bounded and regular subset of  $\Omega_m$ with $K\subset \mathcal{O}$.
By \eqref{estrmp}, we have that $|u_n(t,x)|\le C$ for all $x\in \mathcal{O} $, $t>0$ and $n\ge0$. Since $u_0, (u_{0,n})_{n\ge0} \subset \Bs$ and $u_{0,n} \to u_0$ on $\Bs$, as $n\to \infty$,  we have by \cite[Proposition 3.1  (i), p. 356]{MTW}  that $u_{0,n} \to u_0$ in $\Dd$, we conclude using \cite[Lemma 2.6, p. 89]{CDW} that $g(t) \to u_0$ in $L^1(\mathcal{O})$, as $t\to0$. Therefore, $g(t) \to u_0$ in $L^1_{loc}(\Omega_m)$ as $t\to0$, and from uniqueness of solutions of \eqref{nheq} we have $g\equiv u$ so that the limit $g$ is determined by $u_0$. In particular, it does not depend on the subsequence $(u_{n_k})_{k\ge0}$, so that the whole sequence $(u_n)_{n\ge0}$ converges to $u$ in $C([\tau,\infty),C_0(\Omega_m))$, for every $\tau>0$. This completes the proof.
\end{proof}

%%%%%%%%%%%%%%%%%%%%%%%%%%%%%%%%%
\section{An upper bound on solutions}
\label{upperbound}

In this section we prove Proposition~\ref{preciseupperestimate}.  This proposition is stated for solutions on the domain
$\Omega_m$, but in fact is valid for solutions of \eqref{nheq}, or rather the associated integral equation, on any domain $\Omega$.  Accordingly, we state here the more
general version.  Both the statement and proof are inspired by the statement and proof of \cite[Theorem 1]{W}.
Moreover, we introduce some notation which will be used solely in this section.

Let $\Omega \subset \R^N$ be a domain, not necessarily bounded,  and let
$C_0(\Omega)$ be the space of continuous functions $f:\overline{\Omega} \to \R$ such that $f\equiv 0$ on the boundary $\partial \Omega$ and $f(x) \to 0$ as $|x|\to \infty$ in $\Omega$.
Let $e^{t\Delta}$ be the heat semigroup on $C_0(\Omega)$,  given by a kernel $k_t = k_t^\Omega$ as follows;
\begin{equation}
\label{sgomega}
e^{t\Delta}f(x) = \int_\Omega k_t (x,y) f(y) dy.
\end{equation}
In particular, if $f \in L^1_{loc}(\Omega)$, $f \ge 0$, then $e^{t\Delta}f$ is likewise defined by formula
\eqref{sgomega}.

\begin{theorem}
\label{genupest}
  Fix $\alpha > 0$.
Let $u_0 \in L^1_{loc}(\Omega)$, $u_0 \ge 0$, and suppose that the continuous function $u : (0,T) \to C_0(\Omega)$ is a nonnegative
solution of the integral equation
\begin{equation}
\label{nheqint}
u(t)= e^{t\Delta} u_0 - \int_0^t e^{(t-s)\Delta}\left( u(s)^{\alpha+1}\right) ds.
\end{equation}
It follows that
\begin{equation}
\label{upperestimates4}
u(t,x)\leq \left(\left(e^{t\Delta} u_0(x)\right)^{-\alpha}+\alpha t\right)^{-1/\alpha}
= \frac{e^{t\Delta} u_0(x)}{\left(1 + \alpha t\left(e^{t\Delta} u_0(x)\right)^{\alpha}\right)^{1/\alpha}}
\end{equation}
for all $0 < t < T$ and all $x\in \Omega.$
\end{theorem}

\begin{proof}
Fix $0 < \tau < T$, and set
\begin{equation}
\label{defG}
G(t) = e^{(\tau-t)\Delta}u(t)
= e^{\tau\Delta} u_0 - \int_0^t e^{(\tau-s)\Delta}\left( u(s)^{\alpha+1}\right) ds
\end{equation}
for all $0 \le t \le \tau$.  It is clear from the integral expression in \eqref{defG} that $G:[0,\tau] \to C_0(\Omega)$ is a continuous, decreasing function, with $G(0) = e^{\tau\Delta} u_0$ and $G(\tau) = u(\tau)$.  Furthermore,
 $G:(0,\tau] \to C_0(\Omega)$ is continuously differentiable, and
\begin{equation*}
G'(t) = - e^{(\tau - t)\Delta}\left( u(t)^{\alpha+1}\right)
= -\int_\Omega k_{\tau - t}(\cdot,y)  u(t,y)^{\alpha+1} dy.
\end{equation*}
Since for all $x \in \Omega$ the measure $k_{\tau - t}(x,y)dy$ on $\Omega$, has total mass less than or equal to $1$,
Jensen's inequality implies that
\begin{align*}
G'(t) &=  -\int_\Omega k_{\tau - t}(\cdot,y)  u(t,y)^{\alpha+1} dy\\
& \le - \left(\int_\Omega k_{\tau - t}(\cdot,y)  u(t,y) dy\right)^{\alpha+1}\\
& = - \left( e^{(\tau - t)\Delta}u(t)\right)^{\alpha+1}\\
& = - G(t)^{\alpha+1}.
\end{align*}
Integrating this last differential inequality on $[0,t]$ we obtain
\begin{equation*}
G(t)\leq \frac{1} {\left(G(0)^{-\alpha}+\alpha t\right)^{1/\alpha}},
\end{equation*}
which is the same as
\begin{equation*}
e^{(\tau-t)\Delta}u(t)\leq \frac{1} {\left (\left (e^{\tau\Delta} u_0\right)^{-\alpha}+\alpha t\right)^{1/\alpha}}.
\end{equation*}
This is true for $0 < \tau < T$ and $0 \le t \le \tau$.
The result follows by setting $t = \tau > 0$.
\end{proof}

\begin{remark}
\label{generalisation}
Using an argument similar to the above, one can obtain an analogous estimate for positive solutions of
the more general equation
$$u_t=\Delta u-f(u),$$
where $f$ is a  positive, convex,  increasing $C^2$ function in $(0,\infty)$ such that $F(s)=\int_s^\infty {1\over f(\sigma)}d\sigma<\infty$ for all $s>0.$  Precisely, we have
$$u(t)\leq F^{-1}\left(F(e^{t\Delta} u_0)+t\right),$$
where $F^{-1}$ is the inverse function of $F.$
\end{remark}

%%%%%%%%%%%%%%%%%%%%%%%%%%%%%%%%%
\section{Self-similar asymptotic behavior on sectors}
\label{criticalsec}

In this section we consider equation \eqref{nheq} in the case $2/\alpha=\gamma+m$.
Let $u_0\in \X$ and set $u(t) =  \S_m (t) u_0$. Using \eqref{invpropsmgtime} we can re-write the definition \eqref{omg} of the $\omega$-limits set $\mathcal{Q}^{\gamma}(u_0)$
in the following equivalent form,
\begin{equation}  \label{omgnew}
\mathcal{Q}^{\gamma}(u_0) = \left\{f \in C_0(\Omega_m); \;
\exists \lambda_n \to \infty \ \text{such that}\ \lim_{n\to \infty}  \| \S_m(1)
D_{\lambda_n}^{2/\alpha}  u_0 - f\|_{L^\infty(\Omega_m)} = 0\right\}.
\end{equation}

We begin by proving the Theorem \ref{inihomasysol} which corresponds to the particular case when $\mathcal{Q}^{\gamma}(u_0)$ contains one nontrivial element.

\begin{proof}[\bf Proof of Theorem \ref{inihomasysol}]
Using limits in the sense of $\Dd$, we have
 \[ D_\mu^{\gamma+m} \varphi = D_\mu^{\gamma+m} (\underset{\lambda \to \infty}{\lim}
D_{ \lambda}^{\gamma+m} \psi) =\underset{\lambda \to \infty}{\lim} D_{\mu }^{\gamma+m}
D_{ \lambda}^{\gamma+m} \psi
= \underset{\lambda \to \infty}{\lim} D_{\mu \lambda}^{\gamma+m} \psi =\varphi,\]
for all
$\mu >0$. It follows that $\varphi$  is homogeneous of degree $-(\gamma+m)$.
By uniqueness of solutions of \eqref{nheq}, we deduce that the corresponding solution
 $U(t)=\S_m(t) \varphi$ is self-similar.
By Theorem \ref{depcont}, we have
\[\lim_{\lambda \to \infty}  \S_m(t) D_{\lambda}^{\gamma+m} \psi = \S_m(t) \varphi ,\]
in $C_0(\Omega_m)$, for all $t>0$. From \eqref{invpropsmgtime}, and Theorem~\ref{exist}(iii), we obtain
\[\lim_{\lambda \to \infty}  \Gamma_{\lambda}^{\gamma+m} \S_m(\cdot) \psi = U(\cdot) ,\]
in $C([\tau,t]; C_0(\Omega_m)$, for all $0 < \tau < t$, so that $\S_m(t)\psi$ is
asymptotically self-similar to the self-similar solution
$U(t)$.
\end{proof}

We give now the proof of Theorem \ref{relomg}.

\begin{proof}[\bf Proof of Theorem \ref{relomg}]
 Let $u_0\in \X$ and $M>0$ be such that $M>\| u_0 \|_{ \X}$.
 If $f \in \mathcal S_m(1)  \mathcal{Z}^{\gamma}(u_0)$, there exists $z\in  \mathcal{Z}^{\gamma}(u_0)$
such that $f=  S_m(1) z$. Since $z\in  \mathcal{Z}^{\gamma}(u_0)$,
there exists $ \lambda_n \to \infty$ such that
$D_{\lambda_n}^{{\gamma+m}}u_0 \to z$ in $\Bs$. We deduce,
by Theorem \ref{depcont}, that
$\S_m(1) D_{\lambda_n}^{{\gamma+m}}u_0 \to \S_m(1) z=f$
in $C_0(\Omega_m)$. Then $f \in  \mathcal{Q}^{\gamma}(u_0) $ and so
$\S_m(1) \mathcal{Z}^{\gamma}(u_0) \subset  \mathcal{Q}^{\gamma}(u_0)$.

Conversely, if $f \in  \mathcal{Q}^{\gamma}(u_0)$, then there exists $\lambda_n \to
\infty \ \text{such that}\ \S_m(1) D_{\lambda_n}^{{\gamma+m}}
u_0 \to f$ on $C_0(\Omega_m)$. Since $\Bs$ is compact, there exist a
subsequence $(\lambda_{n_k})_{k\geq1}$ such that $
D_{\lambda_{n_k}}^{{\gamma+m}}  u_0 \to w$ on $\Bs$. Again by Theorem \ref{depcont},
$\S_m(1) D_{\lambda_{n_k}}^{{\gamma+m}}  u_0 \to \S_m(1)
w$ on $C_0(\Omega_m)$, as $k\to \infty$. Therefore $f=  \S_m(1)w \in
\S_m(1) \mathcal{Z}^{\gamma}(u_0)$. This proves the result.
\end{proof}

\begin{proof}[\bf Proof of Corollary \ref{univsol}]
This follows immediately from  \cite[Theorem 1.4, p. 345]{MTW} and Theorem \ref{relomg}.
\end{proof}

%%%%%%%%%%%%%%%%%%%%%%%%%%%%%%%%%
\section{Linear asymptotic behavior on sectors}
\label{linearsec}

In this section, we study the long-time asymptotic behavior of solutions to \eqref{nheq}
in the case  $2/\alpha < \gamma+m$.  The key point is that under the dilations $D_{\sqrt t}^{\gamma+m}$, which preserve
the norm of $\X$, the integral term in \eqref{nheq} decays faster than the difference between the
two other terms.  This is the content of the next proposition.

\begin{proposition} \label{asymlin}
Let $m\in\{1,\cdots,N\}$, $0< \gamma<N$ and let $\alpha>0$ be such that
\[
       \alpha>\frac{2}{\gamma+m}.
\]
Let $u_0\in \X$ and $u(t)=S_m(t)u_0$. It follows that
\[
         D_{\sqrt t}^{\gamma+m} \left(u(t)-e^{t \Delta_m} u_0 \right) \to 0 ,
\]
in $C_0(\Omega_m)$, as $t\to \infty$.
\end{proposition}

\begin{proof}
We know that, for all $t> 0$,
\[
     u(t)-e^{t \Delta_m} u_0 = - \int_0^t e^{(t-s) \Delta_m}\left( |u(s)|^\alpha u(s)\right) ds = -
t\int_0^1  e^{t(1-\sigma) \Delta_m}\left( |u(\sigma t)|^\alpha u(\sigma t)\right) d\sigma .
\]
Therefore, using \eqref{dilsmg}, we have
\begin{eqnarray*}
     D_{\sqrt t}^{\gamma+m} \left(u(t)-e^{t \Delta_m} u_0 \right)
&=& - t\int_0^1  D_{\sqrt t}^{\gamma+m} e^{t(1-\sigma) \Delta_m}\left( |u(\sigma t)|^\alpha u(\sigma t)\right) d\sigma \\
&=& - t\int_0^1   e^{(1-\sigma) \Delta_m}\left(D_{\sqrt t}^{\gamma+m} |u(\sigma t)|^\alpha u(\sigma t)\right) d\sigma .
\end{eqnarray*}
On the other hand, estimating as in \eqref{unalpha}, we see that \[
     |u(\sigma t)|^{\alpha+1} \le\frac{C}{(\sigma t)^{1-\varepsilon}} x_1 \cdots x_m\;(\sigma t+|x|^2)^{-\frac{\gamma'+2m}{2}} ,
\]
for all $0<\varepsilon<\min\left(1, \frac{N-\gamma}{\alpha(\gamma+m)} \right)$,
where $\gamma<\gamma'<N$ satisfies $\gamma-\gamma' = -\alpha\varepsilon(\gamma+m)$. Therefore,
\[
 D_{\sqrt t}^{\gamma+m} |u(\sigma t)|^{\alpha+1} = t^{\frac{\gamma+m}{2}} |u(\sigma t,\sqrt t x)|^{\alpha+1} \le\frac{C}{(\sigma t)^{1-\varepsilon}} t^{\frac{\gamma-\gamma'}{2}} \; x_1\cdots x_m\;  (\sigma +|x|^2)^{-\frac{\gamma'+2m}{2}} .
\]
By Corollary \ref{esttaunew}, we deduce that
\[
e^{(1-\sigma)\Delta_m}  \left( D_{\sqrt t}^{\gamma+m}  |u( \sigma t)|^{\alpha+1} \right)   \le \frac{C}{(\sigma t)^{1-\varepsilon}} t^{\frac{\gamma-\gamma'}{2}} x_1\cdots x_m\;  (1 +|x|^2)^{-\frac{\gamma'+2m}{2}}
\le \frac{C}{(\sigma t)^{1-\varepsilon}} t^{\frac{\gamma-\gamma'}{2}}  .
\]
It follows that
\begin{eqnarray*}
    \left|D_{\sqrt t}^{\gamma+m} \left(u(t)-e^{t \Delta_m} u_0 \right) \right|
\le  t\int_0^1 e^{(1-\sigma) \Delta_m} \left( D_{\sqrt t}^{\gamma+m}  |u( \sigma t)|^{\alpha+1} \right)  d\sigma\\
 \le C t^{\frac{\gamma-\gamma'}{2}+\varepsilon}     \int_0^1 \frac{d\sigma}{\sigma^{1-\varepsilon}}\\
\le C   t^{\varepsilon\left(1-\frac{\alpha(\gamma+m)}{2}\right)}   .
\end{eqnarray*}
Since $\alpha(\gamma+m) >2$, we see that
\[
   D_{\sqrt t}^{\gamma+m} \left(u(t)-e^{t \Delta_m} u_0 \right) \to 0
\]
in $C_0(\Omega_m)$ as $t\to\infty$. This proves the result.
\end{proof}

\begin{proof}[\bf Proof of Theorem \ref{lincase}] The three statements in this theorem follow from Proposition~\ref{asymlin} and, respectively, Corollary 4.2, p. 360, Corollary 1.3, p. 345 and Theorem 1.4, p. 345 in \cite{MTW}.
\end{proof}

%%%%%%%%%%%%%%%%%%%%%%%%%%%%%%%%
\section{Nonlinear asymptotic behavior on sectors}
\label{nonlinsec}

In this section we consider the equation \eqref{nheq} with non-negative initial value $u_0\in\X$
in the case $\alpha<2/(\gamma+m)$, and our goal is to prove Theorem~\ref{nonlincaseCptsector}.
First however, we need to show that the hypothesis on the initial condition $u_0$, which gives a lower bound
for large $|x|$, implies a lower bound on the resulting solution at any fixed positive time.  The key point
is the behavior near the boundary.  We prove the following result.

\begin{proposition}
\label{timepos}
Let $u_0 \in \X$, with $u_0 \ge 0$, and suppose that there exist $\rho > 0$  and $c > 0$ such
that  for all $x \in \Omega_m$ with $|x| \ge \rho$,
\begin{equation}
\label{teeiszero}
u_0(x) \ge c \psi_0(x),
\end{equation}
where $\psi_0$ is given by \eqref{psi0}.
Let $u(t,\cdot) = u(t) = \mathcal S_m(t) u_0$ be the resulting solution of
\eqref{nheq} as given by Theorem~\ref{exist}, and fix any $t_0 > 0$.  It follow that
$v_0 \equiv \mathcal S_m(t_0) u_0$ verifies the condition
\begin{equation}
\label{teezero}
v_0(x) \ge
\begin{cases}
c'x_1x_2\cdots x_m, & 0 < |x| \le 1 ;\\
c'x_1x_2\cdots x_m|x|^{-(\gamma + 2m)}, & |x| \ge 1,
\end{cases}
\end{equation}
for some $c' > 0$, where the constant $c'$ may depend on $t_0$.
\end{proposition}

We refer the reader to \cite{MS} for results of this type on a general domain.
The present situation differs from that in \cite{MS} in that the sector
$\Omega_m$ does not have the required regularity, and also that here we
include the possibility that $u_0$ could be identically zero on a bounded
subset of  $\Omega_m$.  Unlike \cite{MS}, our proof makes use of the explicit form of
the kernel for the heat semigroup on $\Omega_m$.

\begin{proof}
We first note that it suffices by comparison to prove this for the specific initial value
\begin{equation}
\label{cutoffinit}
u_0(x) = \begin{cases}
0, &x \in \Omega_m, |x| < \rho ;\\
c \psi_0(x),  &x \in \Omega_m, |x| \ge \rho,
\end{cases}
\end{equation}
where $\psi_0$ is given by \eqref{psi0}, and $\rho > 0$ is arbitrary. To accomplish this, we first prove that
for any fixed $t_0 > 0$, $v_0 = e^{t_0\Delta_m}u_0$ verifies \eqref{teezero},
where $u_0$ is given by \eqref{cutoffinit}.  For this purpose, since the estimate is linear, the value of
$c > 0$ in \eqref{cutoffinit} is of no importance.

Thus, we consider $e^{t\Delta_m}u_0$ on $\Omega_m$ given by \eqref{sgomg} and \eqref{hksec}, where $u_0$ is given by \eqref{cutoffinit}.
Using the fact that $e^s - e^{-s} \ge 2s$ for all $s \ge 0$, we see that if $x, y \in \Omega_m$ and $1 \le i \le m$ (so that $x_i \ge 0$ and $y_i \ge 0$), then
\begin{equation}
\label{kernelest}
e^{-\frac{|x_i-y_i|^2}{4t}}-e^{-\frac{|x_i+y_i|^2}{4t}}
= e^{-\frac{|x_i|^2}{4t}}e^{-\frac{|y_i|^2}{4t}}\left[e^{\frac{x_iy_i}{2t}}-e^{-\frac{x_iy_i}{2t}} \right]
\ge \left[\frac{x_iy_i}{t}\right]e^{-\frac{|x_i|^2}{4t}}e^{-\frac{|y_i|^2}{4t}}
\ge \left[\frac{x_iy_i}{t}\right]e^{-\frac{|x_i|^2}{2t}}e^{-\frac{|y_i|^2}{2t}}
\end{equation}
In addition, for $m+1 \le j \le N$, we have (since $(s-r)^2 \le 2s^2 + 2r^2$),
\begin{equation*}
e^{-\frac{|x_j-y_j|^2}{4t}}
\ge e^{-\frac{|x_j|^2}{2t}}e^{-\frac{|y_j|^2}{2t}}.
\end{equation*}

It follows that
\begin{align}
\nonumber e^{t\Delta_m}u_0(x) &= \int_{\Omega_m}K_t(x,y) u_0(y) dy
\\
\nonumber&=  (4\pi t)^{-\frac{N}{2}}\int_{\Omega_m} \displaystyle \prod_{j=m+1}^{N}
e^{-\frac{|x_j-y_j|^2}{4t}} \prod_{i=1}^{m}
\left[e^{-\frac{|x_i-y_i|^2}{4t}}-e^{-\frac{|x_i+y_i|^2}{4t}} \right]u_0(y) dy
\\
\nonumber& \ge c  x_1x_2\cdots x_m t^{-m}(4\pi t)^{-\frac{N}{2}}e^{-\frac{|x|^2}{2t}}
\int_{\Omega_m} y_1y_2\cdots y_m e^{-\frac{|y|^2}{2t}} u_0(y) dy
\\
&= cx_1x_2\cdots x_m t^{-m}(4\pi t)^{-\frac{N}{2}}e^{-\frac{|x|^2}{2t}}\int_{\substack{y \in \Omega_m\\|y|\ge \rho}}
y_1^2y_2^2\cdots y_m^2 e^{-\frac{|y|^2}{2t}} |y|^{-\gamma -2m} dy. \label{smallx}
\end{align}

This shows in particular that for any $t > 0$,  $e^{t\Delta_m}u_0$ satisfies \eqref{teezero}, but only on any give bounded set in $\Omega_m$.

We turn our attention to the case where $|x|$ is large.

For $x \in \Omega_m$, let
\begin{equation}
\label{subset}
\Omega_m(x) = \{y \in \Omega_m : 0 < x_i \le y_i \le \max[2x_i,2], 1 \le i \le m, 0 < |x_i| \le |y_i| \le \max[2|x_i|,2], m < i \le N \}.
\end{equation}
If $y \in \Omega_m(x)$, then
\begin{equation*}
|y|^2 = \sum_{i = 1}^N y_i^2 \le \sum_{i = 1}^N \max[2|x_i|,2]^2
\le  \sum_{i = 1}^N (4x_i^2 + 4) = 4|x|^2 + 4N.
\end{equation*}
Since this calculation is for large $|x|$ we may suppose that
\begin{equation}
\label{lowerbd1}
|x|^2 \ge N,
\end{equation}
and so we see that
\begin{equation}
y \in \Omega_m(x) \implies |y| \le 2\sqrt2 |x| \le 4|x|.
\end{equation}

Also, we want to use the specific formula in \eqref{cutoffinit}, so we impose
\begin{equation}
\label{lowerbd2}
|x| \ge \rho,
\end{equation}
where $\rho$ is as in \eqref{cutoffinit}.  Hence
\begin{equation}
y \in \Omega_m(x) \implies |y| \ge \rho.
\end{equation}

We can now calculate.
\begin{align}
e^{t\Delta_m}u_0(x) &= \int_{\Omega_m}K_t(x,y) u_0(y) dy \ge \int_{\Omega_m(x)}K_t(x,y) u_0(y) dy \nonumber
\\
&=\int_{\Omega_m(x)}K_t(x,y) y_1y_2\cdots y_m  |y|^{-\gamma -2m} dy \nonumber
\\
&\ge 4^{-\gamma - 2m}  |x|^{-\gamma -2m} \int_{\Omega_m(x)}K_t(x,y)  y_1y_2\cdots y_mdy \nonumber
\\
&=  4^{-\gamma - 2m}  |x|^{-\gamma -2m}(4\pi t)^{-\frac{N}{2}}
\left(\prod_{i =1}^{m}\int_{x_i}^{\max[2x_i,2]}
\left[e^{-\frac{|x_i-y_i|^2}{4t}}-e^{-\frac{|x_i+y_i|^2}{4t}} \right] y_idy_i\right) \nonumber \\
&\hspace{1cm} \times \left(\prod_{i =1+m}^{N}\int_{{|x_i|}\leq |y_i|\leq {\max[2|x_i|,2]}}
e^{-\frac{|x_i-y_i|^2}{4t}} dy_i\right) \nonumber
\\
&=  4^{-\gamma - 2m}  |x|^{-\gamma -2m}(4\pi t)^{-\frac{N}{2}}
\left(\prod_{i =1}^{m}\int_{x_i}^{\max[2x_i,2]}
\left[e^{-\frac{|x_i-y_i|^2}{4t}}-e^{-\frac{|x_i+y_i|^2}{4t}} \right] y_idy_i\right)  \label{firstest}\\
&\hspace{1cm} \times \left(\prod_{i =1+m}^{N}\int_{|x_i|}^{\max[2|x_i|,2]}
\left[e^{-\frac{|x_i-y_i|^2}{4t}}+e^{-\frac{|x_i+y_i|^2}{4t}}\right] dy_i\right). \nonumber
\end{align}

 We first  need to examine the integral
\begin{equation*}
\int_{x_i}^{\max[2x_i,2]}
\left[e^{-\frac{|x_i-y_i|^2}{4t}}-e^{-\frac{|x_i+y_i|^2}{4t}} \right] y_idy_i
\end{equation*}
for  $1\leq i\leq m,$ under two different circumstances, $0 < x_i < 1$ and $x_i \ge 1$.  Consider first the case $x_i \ge 1$.  We have,
since $y_i \ge x_i$,
\begin{align*}
\int_{x_i}^{\max[2x_i,2]}\left[e^{-\frac{|x_i-y_i|^2}{4t}}-e^{-\frac{|x_i+y_i|^2}{4t}} \right] y_i dy_i
& \ge x_i\int_{x_i}^{2x_i}\left[e^{-\frac{|x_i-y_i|^2}{4t}}-e^{-\frac{|x_i+y_i|^2}{4t}} \right] dy_i
\\
&= x_i\int_{0}^{x_i}\left[e^{-\frac{|y_i|^2}{4t}}-e^{-\frac{|2x_i+y_i|^2}{4t}} \right] dy_i
\\
&\ge x_i\int_{0}^{x_i}\left[e^{-\frac{|y_i|^2}{4t}}-e^{-\frac{|2x_i|^2}{4t}} e^{-\frac{|y_i|^2}{4t}}\right] dy_i
\\
&= x_i\int_{0}^{x_i}e^{-\frac{|y_i|^2}{4t}}\left[1 - e^{-\frac{|2x_i|^2}{4t}} \right] dy_i
\\
&\ge x_i\left[1 - e^{-\frac{1}{t}} \right]\int_{0}^1e^{-\frac{|y_i|^2}{4t}} dy_i
\\
& = C^1_t x_i.
\end{align*}
We next consider the case $x_i \le 1$.  We have, by \eqref{kernelest},
\begin{align*}
\int_{x_i}^{\max[2x_i,2]}\left[e^{-\frac{|x_i-y_i|^2}{4t}}-e^{-\frac{|x_i+y_i|^2}{4t}} \right] y_i dy_i
& \ge \int_1^2\left[\frac{x_iy_i}{t}\right]e^{-\frac{|x_i|^2}{2t}}e^{-\frac{|y_i|^2}{2t}}y_i dy_i
 \\
& = x_ie^{-\frac{|x_i|^2}{2t}}\int_1^2\left[\frac{y_i^2}{t}\right]e^{-\frac{|y_i|^2}{2t}} dy_i
\\
& \ge x_ie^{-\frac{1}{2t}}\int_1^2\left[\frac{y_i^2}{t}\right]e^{-\frac{|y_i|^2}{2t}} dy_i
\\
& = C^2_t x_i.
\end{align*}

We  second need to examine the integral
\begin{equation*}
\int_{|x_i|}^{\max[2|x_i|,2]}
\left[e^{-\frac{|x_i-y_i|^2}{4t}}+e^{-\frac{|x_i+y_i|^2}{4t}} \right] dy_i
\end{equation*}
for $m+1\leq i\leq N,$ under two different circumstances, $0 < |x_i| < 1$ and $|x_i| \ge 1$.

Consider first the case $|x_i| \ge 1$.  We have, if in addition  $x_i<0,$ that is $-x_i\geq 1,$
\begin{align*}
\int_{|x_i|}^{\max[2|x_i|,2]}\left[e^{-\frac{|x_i-y_i|^2}{4t}}+e^{-\frac{|x_i+y_i|^2}{4t}} \right]  dy_i
&=\int_{-x_i}^{-2x_i}\left[e^{-\frac{|x_i-y_i|^2}{4t}}+e^{-\frac{|x_i+y_i|^2}{4t}} \right]  dy_i\\
& \geq \int_{-x_i}^{-2x_i} \left[e^{-\frac{|x_i+y_i|^2}{4t}} \right] dy_i\\
&=\int_{0}^{-x_i}e^{-\frac{|y_i|^2}{4t}}  dy_i\\
& \ge \int_{0}^{1}e^{-\frac{|y_i|^2}{4t}}  dy_i
\\
& = C^3_t .
\end{align*}
We have, if  in addition $x_i>0,$ that is $x_i\geq 1,$
\begin{align*}
\int_{|x_i|}^{\max[2|x_i|,2]}\left[e^{-\frac{|x_i-y_i|^2}{4t}}+e^{-\frac{|x_i+y_i|^2}{4t}} \right]  dy_i
&=\int_{x_i}^{2x_i}\left[e^{-\frac{|x_i-y_i|^2}{4t}}+e^{-\frac{|x_i+y_i|^2}{4t}} \right]  dy_i\\
&\geq \int_{x_i}^{2x_i}\left[e^{-\frac{|x_i-y_i|^2}{4t}} \right]  dy_i\\
&=\int_{0}^{x_i}e^{-\frac{|y_i|^2}{4t}}   dy_i\\
& \ge \int_{0}^{1}e^{-\frac{|y_i|^2}{4t}} dy_i
\\
& = C^3_t .
\end{align*}

We next consider the case $|x_i| \le 1$.  By the inequality, $e^{-\frac{|x_j-y_j|^2}{4t}}
\ge e^{-\frac{|x_j|^2}{2t}}e^{-\frac{|y_j|^2}{2t}},$ we have
\begin{align*}
\int_{|x_i|}^{\max[2|x_i|,2]}\left[e^{-\frac{|x_i-y_i|^2}{4t}}+e^{-\frac{|x_i+y_i|^2}{4t}} \right]dy_i
 &\ge \int_1^2\left[e^{-\frac{|x_i-y_i|^2}{4t}} \right] dy_i \\
& \ge \int_1^2e^{-\frac{|x_i|^2}{2t}}e^{-\frac{|y_i|^2}{2t}} dy_i
\\
& \ge e^{-\frac{1}{2t}}\int_1^2e^{-\frac{|y_i|^2}{2t}} dy_i
\\
& = C^4_t.
\end{align*}

In all cases, we have
\begin{equation}
\label{hardest}
\int_{x_i}^{\max[2x_i,2]}
\left[e^{-\frac{|x_i-y_i|^2}{4t}}-e^{-\frac{|x_i+y_i|^2}{4t}} \right] y_idy_i \ge C_t x_i
\end{equation}
for $1\leq i\leq m$ and
\begin{equation}
\label{hardest2}
\int_{|x_i|}^{\max[2|x_i|,2]}
\left[e^{-\frac{|x_i-y_i|^2}{4t}}+e^{-\frac{|x_i+y_i|^2}{4t}} \right] dy_i \ge C_t
\end{equation}
for $m+1\leq i\leq N,$ whenever $x \in \Omega_m$.

It therefore follows from \eqref{firstest}, \eqref{hardest}, \eqref{hardest2} that, if $x \in \Omega_m$, then
\begin{equation}
\label{largex}
e^{t\Delta_m}u_0(x) \ge C_t x_1x_2 \cdots x_m|x|^{-\gamma -2m}, \, |x| \ge \max[\sqrt N,\rho].
\end{equation}
Combining \eqref{smallx} and \eqref{largex}, we obtain that for any fixed $t > 0$,
$e^{t\Delta_m}u_0$ satisfies \eqref{teezero}.

We next show the same result for $u(t,\cdot) = u(t) = \mathcal S_m(t) u_0$ be the resulting solution of
\eqref{nheq}, where $u_0$ is given by \eqref{cutoffinit}. To do so, set
$w(t) = e^{\mu t}u(t)$, where
$\mu = [c\rho^{\gamma + m}]^\alpha \ge \|u_0\|_{L^\infty(\Omega_m)}^\alpha$.
Since $u(t) \le \|u_0\|_{L^\infty(\Omega_m)}$ for all $t > 0$, we have  $u(t)^\alpha \le \mu$ for all $t > 0$.
It follows that
\begin{equation*}
w'(t) = e^{\mu t}u'(t) + e^{\mu t}\mu u(t) \ge e^{\mu t}u'(t) + e^{\mu t}u(t)^\alpha u(t) = e^{\mu t}\Delta u(t) = \Delta w(t).
\end{equation*}
Hence $w(t) \ge e^{t\Delta_m}w(0) = e^{t\Delta_m}u_0$. In other words $u(t) \ge e^{-\mu t}e^{t\Delta_m}u_0$,
which implies the desired result.
\end{proof}

\begin{remark}
\label{existLinf}
In addition to being well-posed in $C_0(\Omega_m)$, in $L^q(\Omega_m)$ for $1 \le q < \infty$, as noted in the introduction, and in $\X$, as per Theorem~\ref{exist}, equation \eqref{nheq} is globally well-posed in $L^\infty(\Omega_m)$ in the following sense.  For every $u_0 \in L^\infty(\Omega_m)$, there is a unique solution $u \in C((0,\infty); C_0^{b,u}(\Omega_m))$
of the integral equation \eqref{nheqINTE}, where $C_0^{b,u}(\Omega_m)$ denotes the closed subspace of $L^\infty(\Omega_m)$
of bounded, uniformly continuous functions on $\Omega_m$ which are zero on $\partial\Omega_m$, but not necessarily as $|x| \to \infty$. This solution has the following additional properties: the function $u$ is a classical solution of \eqref{nheq} on
$(0,\infty) \times \Omega_m$, $\|u(t) - e^{t\Delta_m}u_0\|_{L^\infty(\Omega_m)} \to 0$ as $t \to 0$, and $|u(t)| \le (\alpha t)^{1/\alpha}$, for all $t > 0$.  One way to see this is first to establish the corresponding result on $L^\infty(\R^N)$, but of course with $C^{b,u}(\R^N)$ instead of $C_0^{b,u}(\Omega_m)$, and then to restrict to anti-symmetric functions on $\R^N$.  The result on $\R^N$ follows from standard arguments, i.e. contraction mapping, parabolic regularity, and comparison.  We refer the reader to Appendices B and C of \cite{CDNW} for detailed information about $e^{t\Delta}$ on $C^{b,u}(\R^N)$.  In particular,
\cite[Lemma B.1]{CDNW} establishes that $e^{t\Delta}h \in C^{b,u}(\R^N)$ for all $h \in L^\infty(\R^N)$ and
\cite[Theorem C.1]{CDNW}, which still valid for the nonlinear heat equation with absorption,  establishes the necessary regularity.
\end{remark}

\begin{proposition}
\label{bigselsim}
Let $m \in \{1, 2, \cdots, N\}$ and $\alpha > 0$.  There exists a self-similar solution
$V(t,x) = t^{-1/\alpha}g(\frac{x}{\sqrt t})$ of equation \eqref{nheq} such that $g \in C_0^{b,u}(\Omega_m)$,
the space of bounded uniformly continuous functions on $\Omega_m$ which are zero on $\partial\Omega_m$, $g \ge 0$, and
\begin{equation}
\label{gbeh}
\alpha^{-1/\alpha}e^{\Delta_m}h \le g \le (\alpha \epsilon)^{-1/\alpha}e^{(1-\epsilon)\Delta_m}h
\end{equation}
for all $0 < \epsilon < 1$, where $h(x) = 1$ is the constant function on $\Omega_m$.

The self-similar solution $V$ is characterized by
\begin{equation}
\label{selsimlim}
V = \lim_{\lambda \to \infty}\Gamma_\lambda^{2/\alpha} v
\end{equation}
where $v$ is the solution to \eqref{nheq} with initial value $v_0 = h$, as described in Remark~\ref{existLinf},
 the dilations $\Gamma_\lambda^{2/\alpha}$ are defined by \eqref{dilsolnew},
and where the
 limit \eqref{selsimlim} is  uniform on compact subsets of $(0,\infty)\times \overline\Omega_m$.
\end{proposition}

We observe that in the case $m = 0$, the corresponding self-similar solution is
$(\alpha t)^{-1/\alpha}$.

\begin{proof}
Throughout this proof, we let $h \in L^\infty(\Omega_m)$ denote the specific function
\begin{equation}
\label{h}
h(x) = 1, \, x \in \Omega_m.
\end{equation}
It follows from
\eqref{dilsmg} that
\begin{equation}
\label{constscale}
(e^{\lambda^2t\Delta_m}h)(\lambda x) = (e^{t\Delta_m}h)(x)
\end{equation}

Next we let $v = v(t,x)$ be the global solution of \eqref{nheq} or  \eqref{nheqINTE} with initial
value $v_0 = h$, i.e. $v_0(x) = v(0,x) = 1$, for all $x \in \Omega_m$, as described in Remark~\ref{existLinf}.
For all $\lambda > 0$,
\begin{equation*}
v_\lambda(t,x) = \lambda^{2/\alpha}v(\lambda^2t, \lambda x)
\end{equation*}
 is likewise a solution of
\eqref{nheq} or  \eqref{nheqINTE}, but with initial value
\begin{equation}
\label{contsiv}
v_{0,\lambda}(x) = v_\lambda(0,x) = \lambda^{2/\alpha}v_0(\lambda x) =  \lambda^{2/\alpha}
\end{equation}
for all $x \in \Omega_m$.  Since $\lambda \to v_{0,\lambda}$ is an increasing function, by comparison
so must be $\lambda \to v_\lambda$.  Moreover, we know that
\begin{equation}
\label{vlambd}
v_\lambda(t,x) \le (\alpha t)^{-1/\alpha},
\end{equation}
so that the $v_\lambda$ must converge to some function
\begin{equation*}
V(t,x) \le (\alpha t)^{-1/\alpha},
\end{equation*}
and in particular  $V(t) \in L^\infty(\Omega_m)$ for $t > 0$. Since each $v_\lambda$ is a solution of the integral equation \eqref{nheqINTE} on every interval
$[\epsilon, T] \subset (0, \infty)$, the same must be true for $V$, by the monotone convergence theorem.
  Hence $V$ is a solution of \eqref{nheq} and $V \in C((0,\infty); C_0^{b,u}(\Omega_m))$ since
initial values in $L^\infty(\Omega_m)$ give rise to
  solutions of \eqref{nheqINTE} in $C((0,\infty); C_0^{b,u}(\Omega_m))$ as per Remark~\ref{existLinf}.
Note that by parabolic regularity and standard compactness arguments, the
convergence of the $v_\lambda$ to $V$ is  uniform on compact subsets of $(0,\infty)\times \overline\Omega_m.$
Moreover, $V$ a self-similar solution, being the limit of the dilated solutions $v_\lambda$.
 Thus we can write
 \begin{equation}
\label{vprofil}
V(t,x) = t^{-1/\alpha}g(\frac{x}{\sqrt t}),
\end{equation}
where $g = V(1) \in C_0^{b,u}(\Omega_m)$ is the profile of $V$.

As for the behavior of $g$ we first observe that, for $t > 0$ and $\epsilon > 0$, by \eqref{vlambd},
 \begin{equation*}
 v_\lambda(t + \epsilon,\cdot)  \le e^{t\Delta_m}(v_\lambda(\epsilon)) \le (\alpha \epsilon)^{-1/\alpha}e^{t\Delta_m}h.
 \end{equation*}
 Letting $\lambda \to \infty$, we see that
for $t > 0$ and $\epsilon > 0$,
\begin{equation*}
V(t + \epsilon) \le e^{t\Delta_m}V(\epsilon) \le (\alpha \epsilon)^{-1/\alpha}e^{t\Delta_m}h,
\end{equation*}
so that
 \begin{equation}
 \label{gupper}
g = V(1) \le (\alpha \epsilon)^{-1/\alpha}e^{(1-\epsilon)\Delta_m}h
\end{equation}
for all small $0 < \epsilon < 1$.
Also
\begin{equation}
\label{Vupper}
V(2t) \le (\alpha t)^{-1/\alpha}e^{t\Delta_m}h.
\end{equation}
On the other hand, we claim that for all $t > 0$
\begin{equation}
\label{Vlower}
V(t) \ge (\alpha t)^{-1/\alpha}e^{t\Delta_m}h.
\end{equation}
To see this, we first show that
\begin{equation}
\label{vlower}
v(t) \ge (1 + \alpha t)^{-1/\alpha} e^{t\Delta_m}h.
\end{equation}
Indeed, if we set $w(t) = (1 + \alpha t)^{1/\alpha}v(t)$, so that $w(0) = v(0) = h$, then since
$v(t) \le (1 + \alpha t)^{-1/\alpha}$, (which follows in particular from Proposition~\ref{preciseupperestimate}
since $|e^{t\Delta_m}v_0| \le 1$) we get that
\begin{align*}
w'(t) &= (1 + \alpha t)^{\frac{1}{\alpha} - 1}v(t) + (1 + \alpha t)^{1/\alpha}v'(t)\\
&\ge (1 + \alpha t)^{1/\alpha}v(t)^{\alpha + 1} + (1 + \alpha t)^{1/\alpha}v'(t)\\
& = \Delta w(t),
\end{align*}
which implies that $w(t) \ge e^{t\Delta_m}w(0) = e^{t\Delta_m}h$.
This proves \eqref{vlower}.  By \eqref{constscale}, it follow that
\begin{equation}
\label{vlamlower}
v_\lambda(t,x) \ge \lambda^{2/\alpha}(1 + \alpha \lambda^2 t)^{-1/\alpha} (e^{t\lambda^2\Delta_m}h)(\lambda x)
 = (\lambda^{-2} + \alpha  t)^{-1/\alpha}e^{t\Delta_m}h.
\end{equation}
The lower bound \eqref{Vlower} now follows by letting $\lambda \to \infty$ in \eqref{vlamlower}.
Hence
\begin{equation}
\label{glower}
g = V(1) \ge \alpha^{-1/\alpha}e^{\Delta_m}h.
\end{equation}

Finally, we note the perhaps curious result that
\begin{equation}
\label{selfcomp}
V(2t) \le (\alpha t)^{-1/\alpha}e^{t\Delta_m}h \le V(t)
\end{equation}
for all $t > 0$.

\end{proof}

\begin{proof}[Proof of Theorem \ref{nonlincaseCptsector}]
 By the hypotheses on $u_0$ and by Proposition~\ref{timepos},   we have that for $t_0>0,$
 \begin{equation}
 \label{u0hat}u(t_0,x)\geq cx_1\cdots x_m\min[1,|x|^{-\gamma-2m}],
  \end{equation}
 on $\Omega_m$ and we know that $u(t_0)\in C_0(\Omega_m).$  Up to a translation in time and since we are concerned with the large time behavior,  we may suppose that $u_0 \in \X\cap C_0(\Omega_m)$, $u_0 \ge 0$ and verifies \eqref{u0hat}.

 In fact, it suffices to assume
 \begin{equation}
 \label{suffice}
 u_0(x) =  cx_1\cdots x_m\min[1,|x|^{-\gamma-2m}].
  \end{equation}
  Indeed, suppose $ u_0(x) =  cx_1\cdots x_m\min[1,|x|^{-\gamma-2m}] \le v_0(x) \le c'$ for some $c' > c$,
  and that $u(t,x)$, $v(t,x)$ and $w(t,x)$ are the solutions of \eqref{nheq}  with initial values respectively
  $u_0$, $v_0$ and $w_0 \equiv c'$.
  We know by comparison that
 \begin{equation}
 \label{pinching}
t^{\frac{1}{\alpha}}u(t,x\sqrt t) \le t^{\frac{1}{\alpha}}v(t,x\sqrt t)
\le t^{\frac{1}{\alpha}}w(t,x\sqrt t).
  \end{equation}
  Hence if we prove that
  \begin{equation*}
  \lim_{t \to \infty} t^{\frac{1}{\alpha}}u(t,x\sqrt t) = g(x)
\end{equation*}
uniformly on compact subsets of $\overline\Omega_m$,
  then clearly,
  since by Proposition~\ref{bigselsim}
    \begin{equation*}
  \lim_{t \to \infty} t^{\frac{1}{\alpha}}w(t,x\sqrt t) = g(x)
\end{equation*}
also uniformly on compact subsets of $\overline\Omega_m$,
it follows that
 \begin{equation*}
  \lim_{t \to \infty}  t^{\frac{1}{\alpha}}v(t,x\sqrt t) = g(x)
\end{equation*}
uniformly on compact subsets of $\overline\Omega_m$.
Thus, we now assume the initial value $u_0 \in \X$ is given by \eqref{suffice}, and we denote by
$u(t)=\mathcal{S}_m(t) u_0$
be the resulting solution of \eqref{nheq} given by Theorem~\ref{exist}.

 We use a method introduced in \cite{KP}. Consider the space-time dilations functions defined by \eqref{dilsolnew} with $\sigma=2/\alpha$:
 \begin{equation}
 \label{ulambda}
 u_\lambda(t,x)=\Gamma^{2/\alpha}_\lambda u(t,x)=\lambda^{2/\alpha}u(\lambda^2 t,\lambda x),\; \lambda>0,\; t>0,\; x\in \Omega_m.
 \end{equation}
 In particular, $u_\lambda$ is the solution of \eqref{nheq} with initial data
  \begin{equation}
 \label{u0lambda}
 u_{0,\lambda}(x)=D^{2/\alpha}_\lambda u_0(x)=\lambda^{2/\alpha}u_0(\lambda x)
 =   c\lambda^{2/\alpha}x_1\cdots x_m\min[\lambda^m,\lambda^{-\gamma - m}|x|^{-\gamma-2m}], \; x\in \Omega_m.
 \end{equation}
  Since $\frac{2}{\alpha} > \gamma + m$, it follows that $ u_{0,\lambda}(x)$ is an increasing function in $\lambda > 0$,
 for all $x\in \Omega_m$.  (It's the minimum of two functions which are obviously increasing in
$\lambda$.)  Consequently, the solutions $u_\lambda(t,x)$ are likewise increasing in $\lambda > 0$.
We note also that the solutions $w_\lambda(t,x)$ are increasing in $\lambda > 0$ (as in the proof of Proposition~\ref{bigselsim}), where $w$ is the solution with initial value $w_0 \equiv c'$ as above.

 Since
 \begin{equation*}
 \label{uperulda}
 u_\lambda(t,x)\leq w_\lambda(t,x) \le V(t,x) , \; \mbox{in}\; (0,\infty)\times \Omega_m,
  \end{equation*}
where $V$ is the self-similar solution of Proposition~\ref{bigselsim}, it follows that the following limit
   \begin{equation}
 \label{limulbda}
\lim_{\lambda\to \infty} u_\lambda(t,x) = U(t,x) \leq V(t,x),
 \end{equation}
 exists and
    \begin{equation}
 \label{limulbdaless}
u_\lambda(t,x) \le U(t,x),
 \end{equation}
for all $\lambda > 0$.  Moreover, by parabolic regularity and standard compactness arguments, the
 limit \eqref{limulbda} is  uniform on compact subsets of $(0,\infty)\times \overline\Omega_m.$

We next wish to show that
  \begin{equation}
 \label{valueU}
 U(t,x)=V(t,x)
 \end{equation}
  on $(0,\infty)\times \Omega_m.$  For this we  need to obtain a lower bound for $U.$

 Let $A>0,$ and consider the family of truncated initial values
\begin{equation*}
 u_{0,\lambda}^A(x)=\min[u_{0,\lambda}(x),A],\; x\in \Omega_m.
\end{equation*}
 Let $z_{\lambda}^A$ be the solution of \eqref{nheq} with initial data $u_{0,\lambda}^A.$ By comparison principle
 and \eqref{limulbdaless}
 \begin{equation}
 \label{compvkuk}
 z_{\lambda}^A(t,x)\leq u_\lambda(t,x) \le U(t,x),\; \mbox{in}\;  [0,\infty)\times \Omega_m,
 \end{equation}
 for every $\lambda>0$ and $A>0.$
 Moreover, it is clear from \eqref{u0lambda} that for each fixed $A > 0$, the initial values
 $u_{0,\lambda}^A(x)$ are an increasing function of $\lambda > 0$, and so therefore
 must be the solutions $z_{\lambda}^A(t,x)$.  Furthermore, the
 initial values satisfy the monotone limit
 \begin{equation}
 \label{inclim}
 \lim_{\lambda \to \infty}u_{0,\lambda}^A(x) = A,
 \end{equation}
 and the corresponding solutions converge in a monotone fashion to some function
  \begin{equation}
 \label{cvvldaA}
 \lim_{\lambda\to \infty} z_{\lambda}^A(t,x)=Z_A(t,x) \le U(t,x).
  \end{equation}

 We next consider the integral equation satisfied by $z_{\lambda}^A(t)$, i.e.
 equation \eqref{nheqINTE} with initial value $u_{0,\lambda}^A$.  Using \eqref{inclim}
 and \eqref{cvvldaA} along with the monotone convergence theorem, we see that
 $Z_A(t)$ satisfies
 \begin{equation}
\label{nheqINTEVA}
 Z_A(t)= e^{t\Delta_m} A - \int_0^t e^{(t-s)\Delta_m}\left( |Z_A(s)|^{\alpha} Z_A(s)  \right) ds,
 \end{equation}
i.e. $Z_A$ is the solution of \eqref{nheqINTE} with initial value $Z_A(0) \equiv A$ on $\Omega_m$.

 We know by (the proof of) Proposition~\ref{bigselsim} that
 \begin{equation*}
 \lim_{A \to \infty} Z_A(t) =V(t),
 \end{equation*}
which implies, along with \eqref{cvvldaA}, that $V(t,x) \le U(t,x)$.
Thus by \eqref{limulbda}, $V(t,x) = U(t,x)$.

Thus we have shown that
   \begin{equation}
 \label{limulbdabis}
\lim_{\lambda\to \infty} u_\lambda(t,x)
= \lim_{\lambda\to \infty} \lambda^{2/\alpha}u(\lambda^2 t,\lambda x)
=  V(t,x) = t^{-1/\alpha}g(\frac{x}{\sqrt t}),
 \end{equation}
 where the limit is uniform on compact subsets of $(0,\infty) \times \overline\Omega_m$.
 The result now follows first by setting $t = 1$ in \eqref{limulbdabis},
 and then by replacing $\lambda^2$ by $\tau$.

\end{proof}

%%%%%%%%%%%%%%%%%%%%%%%%%%%%%%%%%
%%%%%%%%%%%%%%%%%%%%%%%%%%%%%%%%%
\section{Case of $\Rd$}
\label{extensec}

In this section, we consider the extension of the results in the previous section on the sectors $\Omega_m$ to the case of antisymmetric functions on $\Rd$. Recall that if $\psi:\Omega_m \to\R$, then $\widetilde{\psi}$ denotes its
pointwise  extension to $\Rd$ which is antisymmetric with respect to $x_1, x_2,\cdots, x_m$.  Similarly, if $K \subset \overline\Omega_m$,
then $\widetilde K \subset \R^N$ denotes its antisymmetric extension.  Similar notation is used for spaces of functions, etc.

The following two results show the equivalence of various kinds of convergence on $\Omega_m$ to the  corresponding
convergence on  $\R^N$.

\begin{proposition} \label{cvkinds}
Let $m\in\{1,\cdots,N\}$, $0<\gamma<N$ and $M>0$. Let $(\psi_k)_{k\ge1} \subset \B$ and $\psi\in \B$.
The following are equivalent:
\begin{enumerate}
\item[(i)] $\psi_k \to \psi$ in $\Bs$ as $k\to \infty$;
\item[(ii)] $\psi_k \to \psi$ in $\Dd$ as $k\to \infty$;
\item[(iii)] $\widetilde{\psi_k} \to \widetilde{\psi}$ in $\Sd$ as $k\to \infty$;
\item[(iv)] $\widetilde{\psi_k} \to \widetilde{\psi}$ in $\Bts$ as $k\to \infty$;
\item[(v)] $\widetilde{\psi_k} \to \widetilde{\psi}$ in $\mathcal{D}'(\Rz)$ as $\; k\to \infty$;
\item[(vi)] $\widetilde{\psi_k} \to \widetilde{\psi}$ in $(\Bg)^\star$ as $\; k\to \infty$.
\end{enumerate}
\end{proposition}

\begin{proof}
From \cite[Proposition 3.1 (i), p. 356]{MTW} and \cite[Proposition 5.1, p. 361]{MTW} the statements (i), (ii), (iii) and (iv) are equivalent. From \cite[Proposition 2.1 (i), p. 1110]{CDWL} we have (v) and (vi) are equivalent. It is clear that (v) implies (ii) and (iii) implies (v). This proves the result.
\end{proof}

\begin{proposition} \label{cvloc}
Let $m\in\{1,\cdots,N\}$, $0<\gamma<N$ and $M>0$. Let $(\psi_k)_{k\ge1} \subset \B$ and $\psi\in \B$.
The two following statement are equivalent:
\begin{enumerate}
\item[(i)] $\psi_k \to \psi$ in $L^1_{loc}(\Omega_m)$ as $k\to \infty$;
\item[(ii)] $\widetilde{\psi_k} \to \widetilde{\psi}$ in $L^1_{loc}(\Rz)$ as $\; k\to \infty$.
\end{enumerate}
\end{proposition}

\begin{proof}
(i) $\Rightarrow$ (ii). It suffices to show that
$\int_{{\rho\le |x| \le R}}|\widetilde{\psi_k}-\widetilde{\psi}|dx\to 0$
for all $0 < \rho < R < \infty$.  We know by assumption (i) that for
every $\delta > 0$,
$\int_{\widetilde{K_\delta}}|\widetilde{\psi_k}-\widetilde{\psi}|dx = 2^m\int_{K_\delta}|\psi_k-\psi|dx\to 0$ as $k\to \infty$,
where $K_\delta = \{x \in \Omega_m : \rho \le |x| \le R, \, \rm{dist}(x,\partial\Omega_m) \ge \delta\}$.
On the other hand, $\int_{K_\delta^c}|\psi_k-\psi|dx \le 2M\int_{K_\delta^c}\psi_0dx \to 0$ as $\delta \to 0$,
where $K_\delta^c = \{x \in \Omega_m : \rho \le |x| \le R, \, \rm{dist}(x,\partial\Omega_m) \le \delta\}$.
Thus, given $\epsilon > 0$, fix $\delta > 0$ so that $\int_{K_\delta^c}|\psi_k-\psi|dx \le \frac{\epsilon}{2^{m+1}}$
for all $k \ge 1$
and then choose $k_0 > 0$ so that $\int_{K_\delta}|\psi_k-\psi|dx \le \frac{\epsilon}{2^{m+1}}$ for all $k \ge k_0$.

(ii) $\Rightarrow$ (i).  Let $K$ be a compact of $\Omega_m.$ Then by continuity of the reflection function, $\widetilde{K}$ is a compact of $\Rz$ and $\int_{K}|\psi_k-\psi|dx= 2^{-m}\int_{\widetilde{K}}|\widetilde{\psi_k}-\widetilde{\psi}|dx\to 0$ as $k\to \infty.$ Hence (i) holds. This establishes the result.
\end{proof}

In light of Propositions~\ref{cvkinds} and \ref{cvloc},  Theorems~\ref{extexist}, \ref{extdepcnt}, \ref{extencritical}, \ref{extenlincase},
and \ref{nonlincaseRn} are now immediate consequences of the analogous results on the sector $\Omega_m$, either by
re-interpretation as results about antisymmetric functions on $\R^N$ as described in \cite[Section 3]{TW2}, or by simply re-doing the proofs essentially line for line
but considering the antisymmetric extension to $\R^N$ of all the functions defined on $\Omega_m$.

We wish, however, to specifically identify the self-similar solution on $\R^N$ which is the antisymmetric extension
of the self-similar solution constructed in Proposition~\ref{bigselsim}, as we think it is of sufficient independent interest.

\begin{proposition}
\label{newselsim}
Let $m \in \{1, 2, \cdots, N\}$ and $\alpha > 0$.  There exists a self-similar solution
$V(t,x) = t^{-1/\alpha}g(\frac{x}{\sqrt t})$ of equation \eqref{nheq} such that $g \in C^{b,u}(\R^N)$,
the space of bounded uniformly continuous functions on $\R^N$, $g$ is anti-symmetric in $x_1, x_2, \cdots, x_m$, and
\begin{equation}
\label{gbehRn}
\alpha^{-1/\alpha}e^{\Delta}h(x) \le g(x) \le (\alpha \epsilon)^{-1/\alpha}e^{(1-\epsilon)\Delta}h(x), \quad x \in \Omega_m,
\end{equation}
for all $0 < \epsilon < 1$, where $h \in L^\infty(\R^N)$ is the antisymmetric function such that  $h(x) = 1, x \in \Omega_m$.

The self-similar solution $V$ is characterized by
\begin{equation}
\label{selsimlimRn}
V = \lim_{\lambda \to \infty}\Gamma_\lambda^{2/\alpha} v
\end{equation}
where $v$ is the solution to \eqref{nheq} on $\R^N$ with initial value $v_0 = h$, as described in Remark~\ref{existLinf},
 the dilations $\Gamma_\lambda^{2/\alpha}$ are defined by \eqref{dilsolnew},
and where the
 limit \eqref{selsimlim} is  uniform on compact subsets of $(0,\infty)\times \R^N$.
\end{proposition}

\section{Appendix}

We give here the proof of the parabolic version of the Kato's inequality,  and we use it to establish a basic estimation used to prove Theorem 1.1. See also \cite[Lemma A.1, p. 570]{Oswald}.

\begin{lemma} \label{kato} (Kato's parabolic  inequality)
Let $Q \subset \R \times \R^N$ be any open set. Let $u \in L^1_{loc}(Q)$ be such that:
\[
u_t -\Delta u = f \quad {\rm  in } \ \mathcal{D}'(Q);
\]
with $f \in L^1_{loc}(Q)$, then
\[
|u|_t - \Delta |u| \leq {\rm sign} (u)  f \ \ {\rm  in } \ \mathcal{D}'(Q) .
\]
where
$$  {\rm sign} (u)=\left\lbrace
\begin{array}{ll}
1 & \mbox{if $u>0,$}\\
-1 & \mbox{if $u<0,$}\\
0 & \mbox{if $u=0.$}
\end{array}
\right.$$

\end{lemma}

\begin{proof}
If $F : \R \to \R$ is a $C^2$ convex function and $z:Q\to \R$ a $C^2$ function, then
\begin{eqnarray*}
 (\partial_t - \Delta) F(z) &=&  F'(z)\partial_t z - \left[ F'(z) \Delta z + F''(z) |\nabla z|^2\right] \\
 &= & F'(z)(\partial_t - \Delta) z - F''(z) |\nabla z|^2\\
&\leq & F'(z)(\partial_t - \Delta) z.
\end{eqnarray*}
Mollify $u$ to $u_k = \rho_k \star u$ such  that $u_k\in C^\infty(Q)$, where $\rho_k$ is a sequence of mollifiers. Note that $u_k \to u$ and $ (\partial_t - \Delta) u_k \to  (\partial_t - \Delta)u$ in $L_{loc}^1(Q)$ as $k\to \infty$. It follows that
\[  (\partial_t - \Delta) F(u_k)\leq F'(u_k)(\partial_t - \Delta) u_k.\]
We set $F(z) =\sqrt{\varepsilon^2+z^2}$. We obtain then
\begin{equation} \label{mm}
   (\partial_t - \Delta) F(u_k)\leq \frac{u_k}{F(u_k)}(\partial_t - \Delta) u_k.
\end{equation}
By a simple calculation, we have
\[
    |F(u_k) -F(u)| \leq \big||u_k|-|u|\big| \leq |u_k-u|
\]
then $F(u_k) \to F(u)$, as $k\to \infty$, in $L_{loc}^1(Q)$ as well as pointwise a.e.  and
$(\partial_t - \Delta) F(u_k) \to  (\partial_t - \Delta)F(u)$, as $k\to \infty$, in $\mathcal{D}'(Q)$.
Since $\left| \frac{u_k}{F(u_k)} \right| \leq 1$ then the dominated convergence theorem implies that $\frac{u_k}{F(u_k)} \to \frac{u}{F(u)}$, as $k\to \infty$, in $L_{loc}^1(Q)$ .
Letting $k\to \infty$ in \eqref{mm}, we obtain that
\begin{equation} \label{mmeps}
  (\partial_t - \Delta) F(u)\leq\frac{u}{F(u)}(\partial_t - \Delta) u =\frac{u}{F(u)} f .
\end{equation}
Since $F(u) \to |u|$ uniformly, as  $\varepsilon \to 0$, such that $ (\partial_t - \Delta) F(u)\to (\partial_t - \Delta) |u|$ in $\mathcal{D}'(Q)$. Also $\frac{u}{F(u)} \to \frac{u}{|u|}$ in
$L_{loc}^1(Q)$ (again by the dominated convergence theorem). By letting $\varepsilon \to 0$ in
\eqref{mmeps}, we  obtain that
\[
     (\partial_t - \Delta) |u| \leq {\rm sign} (u)  f.
\]
This completes the proof.
\end{proof}

We have the following result, which is an application of Kato's inequality.

\begin{corollary}\label{cki} Let $X= C_0(\Omega_m)$ or $L^p(\Omega_m)$ for some $1\le p <\infty$. Let $u,v\in C((0,\infty),X)$  be two solutions of the equation \eqref{nheq} with initial values respectively
$ u_0,v_0 \in X$. Then
\[
  |u(t)-v(t)| \leq e^{t\Delta_m} |u_0-v_0|,
\]
for all $t>0$.
\end{corollary}

\begin{proof} Denote by $w$ the unique solution with initial value $w_0=\frac{|u_0-v_0|}{2}\in X$.\\
Let $z=|u-v|$. Applying Lemma \ref{kato} with $Q=(0,\infty)\times\Omega_m$ and $f=|v|^\alpha v - |u|^\alpha u\in C(Q)$, we have that
\[
   z_t - \Delta z + \left| |u|^\alpha u- |v|^\alpha v \right| \le 0.
\]
Since, $\left| |u|^\alpha u- |v|^\alpha v \right|\geq 2^{-\alpha} |u-v|^{\alpha+1}=2^{-\alpha} z^{\alpha+1}$, we deduce that
\[
    z_t - \Delta z + 2^{-\alpha} z^{\alpha+1} \leq 0.
\]
Let $\overline{z} = \frac{z}{2}$. Then
\[
     \overline{z}_t - \Delta \overline{z} +  \overline{z}^{\alpha+1} \leq 0 = w_t - \Delta w +  w^{\alpha+1}.
\]
Since $\overline{z}(0)=w(0)$, it follows from the comparison principle, that $ \overline{z} \leq w$.
Since
$$
  w(t) = e^{t\Delta_m} w_0 - \int_0^t e^{(t-s)\Delta_m} \left(w(s)^{\alpha+1} \right) ds \leq e^{t\Delta_m} w_0,
$$
the result follows.
\end{proof}

Finally, we give two results which we found during our research
for this article, and which we believe have an independent interest, but which
ultimately were not needed for the proofs of the main results.

Consider the eigenvalue problem, on some domain $B \subset \R^N$
\begin{equation}
\label{ev-a}
-\Delta H = \Lambda H
\end{equation}
where $\Lambda \in \R$. We look for a solution of the form
\begin{equation}
\label{ev-b}
H(x) = x_1x_2\cdots x_mQ(r)
\end{equation}
where $r = (x_1^2 + x_2^2 + \cdots + x_N^2)^{1/2}$.
We note that for $1 \le i \le m$,
\begin{eqnarray*}
\partial_i H(x) &=& x_1x_2\cdots\hat x_i\cdots x_mQ(r) + x_1x_2\cdots x_mQ'(r)\frac{\partial r}{\partial x_i}\\
&=& x_1x_2\cdots\hat x_i\cdots x_mQ(r) + x_1x_2\cdots x_mQ'(r)\frac{x_i}{r},\\
\end{eqnarray*}
where $\hat x_i$ means that $x_i$ is missing from the product, and
\begin{eqnarray*}
\partial_i^2 H(x) &=& 2x_1x_2\cdots\hat x_i\cdots x_m Q'(r)\frac{x_i}{r}
 + x_1x_2\cdots x_m\left[Q''(r)\left(\frac{x_i}{r}\right)^2 + Q'(r)\frac{r^2 - x_i^2}{r^3}\right]\\
 &=& 2x_1x_2\cdots x_m\frac{Q'(r)}{r}
 + x_1x_2\cdots x_m\left[Q''(r)\left(\frac{x_i}{r}\right)^2 + Q'(r)\frac{r^2 - x_i^2}{r^3}\right],\\
\end{eqnarray*}
and if $m < i \le N$, then
\begin{equation*}
\partial_i^2 H(x) =
 x_1x_2\cdots x_m\left[Q''(r)\left(\frac{x_i}{r}\right)^2 + Q'(r)\frac{r^2 - x_i^2}{r^3}\right],
\end{equation*}
It follows that
\begin{align}
\Delta H(x)
&= \sum_{i = 1}^N \partial_i^2 H(x)
 =  x_1x_2\cdots x_m \left[\frac{2m}{r}Q'(r) + Q''(r) +  \frac{N-1}{r}Q'(r) \right] \nonumber\\
&= x_1x_2\cdots x_m \left[Q''(r) + \frac{N + 2m - 1}{r}Q'(r)\right]. \label{ev-c}
\end{align}

\begin{proposition}
\label{antisymef}
Let $B_1 = \{x \in \Omega_m : |x| < 1\}\subset \R^N$, and let $\Lambda > 0$ be the lowest eigenvalue
of $-\Delta$ on $B_1$ with Dirichlet boundary conditions.  It follows that there exists
an eigenfunction $H_1: \overline{B_1} \to [0, \infty)$ of the form \eqref{ev-b}
where $Q: [0,1] \to [0, \infty)$ is decreasing with $Q(0)=1$ and $Q(1) = 0$
and $r = (x_1^2 + x_2^2 + \cdots + x_N^2)^{1/2}$.
Moreover, the value of $\Lambda > 0$ is precisely the lowest eigenvalue of
$-\Delta$ on the unit ball in $\R^{N + 2m}$ with Dirichlet boundary conditions,
and its corresponding eigenfunction is precisely $Q(r')$ where
$r' =  (x_1^2 + x_2^2 + \cdots + x_{N + 2m}^2)^{1/2}$.
\end{proposition}

\begin{proof}
Let $Q(r')$, where $r' = (x_1^2 + x_2^2 + \cdots + x_{N + 2m}^2)^{1/2}$, denote the radially symmetric, radially decreasing, nonnegative eigenfunction
of $-\Delta$ on the unit ball in $\R^{N + 2m}$, [normalized so that $Q(0) = 1$], with eigenvalue $\Lambda > 0$.
In particular, the function $Q: [0,1] \to [0, \infty)$ satisfies the differential equation
\begin{equation}
\label{ev-d}
-\left[Q''(s) + \frac{N + 2m - 1}{s}Q'(s)\right] = \Lambda Q(s), \quad 0 < s \le 1.
\end{equation}
Let $H : \overline{B_1} \to \R^+$ be given by \eqref{ev-b}, where
$r = (x_1^2 + x_2^2 + \cdots + x_N^2)^{1/2}$.
It follows from
\eqref{ev-c} and \eqref{ev-d} that $-\Delta H = \Lambda H$ on $B_1$ and that
$H(x) = 0$ for all $x \in \partial B_1$.  Since $H(x) > 0$
for all $x \in B_1$, it follows that $\Lambda$ is the lowest eigenvalue
of $-\Delta$ on $B_1$.
\end{proof}

Let us now give a remark about the elliptic equation verified by $\psi_0$.
\begin{remark}
\label{eqpsi0}
Let $N\geq 1,\; m\in \{0,1,\cdots, N\},\; 0<\gamma<N$ and $\psi_0$ be given by \eqref{psi0}. Then
\begin{equation}
\label{lestimates}
-\Delta \psi_0=(\gamma+2m)(N-2-\gamma){\psi_0\over |x|^2},
\end{equation}
for all $x\in \Omega_m.$
\end{remark}
\begin{proof}
By \eqref{psi0} the function $\psi_0$ can be written  in the form \eqref{ev-b}, that is
 $$\psi_0(x)=x_1\cdots x_mQ(r),$$ with $Q(r)=c_{m,\gamma}r^{-\gamma-2m},\; r=(x_1^2+\cdots +x_m^2+\cdots +x_N^2)^{1/2},$ where $c_{m,\gamma}=\gamma(\gamma+2)\cdots(\gamma+2m-2).$  For such a $Q$ we have
\begin{eqnarray*}
Q''(r) + \frac{N + 2m - 1}{r}Q'(r)&=&c_{m,\gamma}(\gamma+2m)[\gamma+2m+1-(N+2m-1)]r^{-\gamma-2m-2}\\&=&c_{m,\gamma}(\gamma+2m)(\gamma+2-N)r^{-\gamma-2m-2}.
 \end{eqnarray*}
 The result follows then by \eqref{ev-c}.
\end{proof}

\end{document}